\documentclass[a4paper]{article}
\usepackage[utf8]{inputenc}
\usepackage[backend=bibtex,style=numeric,maxbibnames=30]{biblatex}
\usepackage{hyperref}

\usepackage{soul}
\usepackage[margin=1in]{geometry}
\usepackage[utf8]{inputenc}
\usepackage{algorithm}
\usepackage{algpseudocode}
\usepackage{amsmath}
\usepackage{amssymb}
\usepackage{amsthm}
\usepackage{tikz-cd}
\usepackage[textsize=scriptsize]{todonotes}
\usepackage{xcolor}
\usepackage{tikz}
\usepackage{enumitem}
\usepackage{hyperref}
\usepackage{cleveref}
\usetikzlibrary{cd, calc, shapes, decorations.shapes, decorations.markings, decorations.pathreplacing, shapes.geometric, shapes.symbols,positioning,arrows}
\usepackage{pgfplots}
\usepackage{float}
\usepackage{url}
\usepackage[normalem]{ulem} 
\usepackage{appendix}
\usepackage{caption}
\usepackage{subcaption}
\usepackage{graphicx}
\usepackage{thm-restate}
\usepackage{comment}
\usepackage{enumitem}

\newtheorem{theorem}{Theorem}[section]

\newtheorem{proposition}[theorem]{Proposition}
\newtheorem{lemma}[theorem]{Lemma}

\newtheorem{corollary}[theorem]{Corollary}

\theoremstyle{definition}
\newtheorem{definition}[theorem]{Definition}
\newtheorem{remark}[theorem]{Remark}
\newtheorem*{remark*}{Remark}
\newtheorem{notation}[theorem]{Notation}

\newtheorem{example}[theorem]{Example}

\newtheorem*{property*}{Property}
 
\newcommand{\closure}[1]{\overline{#1}}
\newcommand{\lub}{\mathrm{l.u.b.}}
\newcommand{\glb}{\mathrm{g.l.b.}}

\newcommand{\NN}{\mathbb{N}}
\newcommand{\RR}{\mathbb{R}}

\newcommand{\F}{\mathbb{F}}

\newcommand{\KK}{\mathcal{K}}
\newcommand{\rank}{\mathrm{rank\,}}

\DeclareMathOperator*{\argmin}{\arg\!\min}
\DeclareMathOperator*{\argmax}{\arg\!\max}

\tikzset{twosimp/.style={fill opacity=0.6,fill=gray,draw opacity=0.9}}

\DeclareMathOperator{\h}{H}
\DeclareMathOperator{\push}{push}
\DeclareMathOperator{\dgm}{dgm}
\DeclareMathOperator{\cost}{cost}

\title{Computing the Matching Distance of $2$-Parameter Persistence Modules from Critical Values}
\author{
  Asilata Bapat\footnote{Australian National University, Canberra, Australia}\\
  \texttt{asilata.bapat@anu.edu.au}
  \and Robyn Brooks \footnote{Boston College, Massachusetts, United States}\\
  \texttt{robyn.brooks@utah.edu} \\
  \and Celia  Hacker\footnote{EPFL, Lausanne, Switzerland} \\
  \texttt{celia.hacker@epfl.ch}\\
  \and Claudia Landi\footnote{ Università di Modena e Reggio Emilia, DISMI, Italy} \\ \texttt{claudia.landi@unimore.it} \\
  \and Barbara I. Mahler \footnote{KTH Royal Institute of Technology Stockholm, Sweden}
  \\ \texttt{bmahler@kth.se}
  \and Elizabeth R. Stephenson \footnote{Institute of Science and Technology Austria, Klosterneuburg, Austria}
  \\ \texttt{elizasteprene@gmail.com}
}

\begin{document}
\maketitle
    
\begin{abstract}    
The exact computation of the matching distance for multi-parameter persistence modules is an active area of research in computational topology. Achieving an easily obtainable exact computation of this distance would permit multi-parameter persistent homology to be a viable option for data analysis. For this purpose, two approaches are currently available, limited to persistence with parameters from $\RR^2$: authors of \cite{Kerber-Lesnick-Oudot2018,Bjerkevik2021} work in the discrete setting and apply the point-line duality; authors of \cite{ethier2023geometry, frosini2023} work in the smooth setting while remaining in the primal plane. In this paper, we streamline the computation of the matching distance in the combinatorial setting while staying in the primal plane. In doing so, besides connecting results from the literature, we give explicit formulas for the switch points needed by all the available methods and we show that it is possible to avoid considering vertical and horizontal lines. For the latter, lines with slope 1 play an essential role. 
\end{abstract}

\section{Introduction}\label{sec:introduction}

Persistent Homology is known as one of the most-often used tools in Topological Data Analysis, allowing one to use topological invariants to understand the underlying shape of data. In particular, single parameter persistence yields a summary of data through a one-dimensional filtration, allowing an overview of the data at many different scales.  Single parameter persistent homology has been the subject of much study and has proven to be useful in many applications \cite{bendich2016persistent,bhattacharya2015persistent,green2019topology,lee2017quantifying,sinhuber2017persistent}.


 However, some data requires filtration along multiple parameters to fully capture its information: this is the role of multi-parameter persistent homology \cite{Carlsson2009,Carlsson}.  In some contexts, it can be helpful to use multiple parameters to capture the details of the data  \cite{Blumberg2020,  Keller2018, Miller2017Data, Riess2020, Vipond2021}.  Additionally, single parameter persistent homology is not robust to outliers in a point cloud; these outliers can lead to a misinterpretation of the persistent homology, with the unnecessary creation or destruction of homology classes. Multi-parameter persistence can be a natural fix to this problem by adding a dimension depending on the density of the samples (see, e.g., \cite{FroLan2013, blumberg2022stability}). Unfortunately, understanding, visualizing, and computing invariants in multi-parameter persistent homology remains a difficult task both mathematically and computationally. This difficulty holds as well when it comes to computing distances between such invariants. 



In the single parameter case, there are several ways to compare persistent homology modules --- such as the bottleneck distance and Wasserstein distances --- which exhibit some stability properties with respect to variations in the input \cite{Cohen-Steiner2007}. For more than one parameter, there are also several definitions of distances between persistence modules. Amongst them, the \emph{matching distance} \cite{Biasotti2008} attracts the attention of multi-parameter persistence practitioners.

To compute the matching distance between two $n$-parameter persistence modules, one uses the fact that restricting an $n$-parameter filtration to any affine line $L$ of positive slope through the parameter space yields a $1$-parameter filtration. There is therefore a corresponding restriction of the $n$-parameter persistence module $M$ to a $1$-parameter module $M^L$ along that line. This construction allows for the knowledge and computational methods from the $1$-dimensional case to be applied to the $n$-dimensional case. 

Indeed, following this idea, the matching distance is defined as a supremum of the one-dimensional bottleneck distance, over the collection of all lines of positive slope in the parameter space, i.e.,  
\[ 
\displaystyle d_{match}(M, N) := \displaystyle \sup_{L\colon u=s\vec m+b} \hat m^L \cdot d_B(\dgm M^{L}, \dgm N^{L}), 
\]
where $\hat m^L$ is a normalization term explained in \Cref{sec:background}. However,  exactly computing this distance is not an easy task given the nature of its definition. As a first step towards an exact computation, several approximations of this distance have been provided for $2$-dimensional persistence modules \cite{BIASOTTI20111735,Cerri2020,kerber2020efficient}.

While approximated computations can be very useful in practice, computing the exact value of the matching distance has its own merits. One of them is that it can distinguish modules for which the rank invariant is a complete invariant, such as rectangle decomposable modules, as the matching distance is a metric for the rank invariant \cite{BotnanLeboviciOudot2022}.

Methods for the exact computation of the matching distance are currently only available for $2$-dimensional modules. One approach, carried out in \cite{ethier2023geometry, frosini2023}, is dedicated to smooth input data and is based on the concept of an extended Pareto grid. Another approach, carried out in    \cite{Kerber-Lesnick-Oudot2018}, is dedicated to combinatorial input and is based on looking at lines as points in the dual plane.  The paper \cite{Bjerkevik2021} presents improvements to this approach in terms of time complexity, still exploiting the point-line duality in the plane. 


In this paper, we streamline the computation of the matching distance for combinatorial input data in the primal plane. Our primary goal for this is to understand the role of lines with special slopes like diagonal, vertical, and horizontal lines. As a secondary goal, we aim to prove that the matching distance in the combinatorial setting is realized as a maximum. In doing so, we also re-interpret results from different papers using the primal plane as a reference.   

The present work is based on a refinement of the framework developed in \cite{Bapat2022}, in which we compute the rank invariant of a multi-persistence module from a finite set of critical parameter values.  These critical parameter values capture all the changes in homology occurring throughout the multi-filtration.  They may also be used to partition the set of positive lines in the parameter space into equivalence classes, where each equivalence class maintains the same birth and death ordering within the restricted module for all possible homology classes.


Based on these results, one sees that the critical values must be relevant to the choice of lines for the computation of the matching distance. Leveraging the framework of \cite{Bapat2022} to compute the matching distance allows us to reduce the number of lines necessary to compute the matching distance to a finite set, thus reducing the computation to a maximum rather than a supremum without exploiting the point-line duality used in \cite{Kerber-Lesnick-Oudot2018}.  Through some examples, we show that considering only lines passing through pairs of critical values is not sufficient. This is because the definition of matching distance depends on the bottleneck distance, and lines in the same equivalence class might not be such that the bottleneck distance is always achieved by the same pairing of births/deaths, as will be further explained in \Cref{sec:mainresult}.  

To overcome this problem, we analyze where switches might happen in the matching that achieves the bottleneck distance, identifying a set of points in the parameter space that we later on refer to as the set of switch points. This set allows us to refine our equivalence relation on the set of positive lines (see \Cref{sec:switch}) by partitioning lines through the parameter space using the union of the critical parameter values with the switch points.  Using this union, it is possible to identify all the lines at which the matching distance can be realised. 

We explain in detail how to compute all relevant parameter values in the primal plane and show that the matching distance is attained either at: a line through a pair of points in this union, or a line of diagonal slope through exactly one of the points in the union (see \Cref{thm:main}). 
Where appropriate, we also detail the connections and differences between related papers.

In conclusion, the novelty of this paper is to provide:  a geometric understanding of different lines, including horizontal, vertical, and diagonal lines, and their contribution to the matching distance; explicit formulas for the switch points; a common  framework where existing approaches in the literature can be compared.    

The paper is structured as follows. In \Cref{sec:background}, we give necessary background from multi-parameter persistent homology, as well as some of the concepts coming from the framework developed in \cite{Bapat2022}. We also provide some examples of two-parameter persistence modules for which we explicitly show which lines are necessary to compute the matching distance. In \Cref{sec:mainresult}, we provide the statement and proof of our main theoretical result \Cref{thm:main} and include a discussion on the role of vertical and horizontal lines. In \Cref{sec:switch}, we provide guidelines for the computation of the switch points, which are necessary to the exact computation of the matching distance. Finally, in \Cref{sec:conclusions}, we discuss future directions of study.



\section{Background}\label{sec:background}
While we will subsequently specialize to the case of $n = 2$ parameters, we begin with some general definitions to set the scene.  
\subsection{Notation and definitions}\label{sec:notation}

\begin{definition}[Parameter Spaces]
Let $n\in \NN$. If $n=1$, endow $\RR$ with the usual order $\le$. If $n>1$, 
take the following partial order on $\RR^n$: for $u=(u_i),v=(v_i)\in\RR^n$,  set $u\preceq v$ if and only if $u_i\leq v_i$ for  $i=1,\ldots, n$. The poset $(\RR^n,\preceq)$ is called an {\em $nD$ parameter space}. 
\end{definition}

For $n>1$, the $nD$ parameter space can be sliced by lines with positive slope.

 A line $L\subset\RR^n$  is a $1D$ parameter space when considered parameterized by $s\in\RR$ as $L: u =\vec m s+b$ where $\vec m\in \RR^n$ and  $b\in \RR^n$. $L$ has positive slope if  each coordinate $m_i$ of $\vec m$ is positive: $m_i > 0$.  Throughout the paper, we consider the following {\em standard normalisation} for parameterizations of lines: 

\begin{align}
\lVert \vec m \rVert_\infty := \max \{ m_i\mid i=1,\ldots, n\} = 1,  & \quad \sum_{i=1}^n b_i = 0  \label{eq:normalization}
\end{align}

This normalization is the one used in papers such as \cite{Cerri-et-al2013,Landi2018}. Other choices are possible. Changing the normalization would impact all the intermediate computations that follow but not the overall results.

\begin{definition}[Persistence Modules] Let $\F$ be a fixed field. An {\em ($n$-parameter) persistence module} $M$ over the parameter space $\RR^n$ is an assignment of an $\F$-vector space $M_u$ to each $u\in\RR^n$, and linear maps, called {\em transition} or {\em internal maps}, $i_M(u,v)\colon M_u\to M_v$ to each pair of points $u\preceq v\in\RR^n$, satisfying the following properties:

\begin{itemize}
    \item $i_M(u,u)$ is the identity map for all $u \in\RR^n$.
    \item $i_M(v,w)\circ i_M(u,v)=i_M(u,w)$ for all $u\preceq v\preceq w \in\RR^n$.
\end{itemize}
\end{definition}
  
In this paper, persistence modules will always assumed to be finitely presented. To fix the ideas, we will assume that they are obtained applying homology  in a fixed degree, say $q$, over a fixed coefficient field $\F$, to a tame and one-critical (cf.~\cite{Bapat2022}) $n$-parameter filtration  $\KK$ of a finite simplicial complex: $M=H_q(\KK;\F)$. 

In the particular case when $n=1$, a finitely presented persistence module $M$ can be uniquely decomposed as a finite sum of interval modules \cite{Zomordian-Carlsson2005}.

\begin{definition}[Interval Module]
 An {\em interval module} is a $1$-parameter persistence module of the form $I[b,d)$ with $b<d\le \infty\in \overline{\RR}:=\RR\cup\{\infty\}$ such that, for every $s\le  t\in \RR$,
\[I[b,d)_s=\left\{\begin{array}{cc}\F & \text{if } b\le s<d\\ 0 & \text{otherwise}\end{array}\right. ,\quad I[b,d)(s\le t)=\left\{\begin{array}{cc}\mathrm{id_\F} & \text{if } b\le s\le t<d\\ 0 & \text{otherwise.}\end{array}\right.\]
 \end{definition}
 
The interval $[b,d)$ can be represented as a point $(b,d)$ in $\RR\times \overline{\RR}$, above the diagonal. This encoding permits defining persistence diagrams.

\begin{definition}[Persistence Diagram]
If $M\cong \bigoplus_{j\in J} I[b_j,d_j)^{m_j}$, then the {\em persistence diagram} of $M$, denoted $\dgm M$, is the finite multi-set of points $(b_j,d_j)$ of $\RR\times \overline{\RR}$ with multiplicity $m_j$ for $j\in J$.
\end{definition}

We now review the definitions of bottleneck distance between persistence diagrams and matching distance between two $n$-parameter persistence modules $M$ and $N$.
 \\

 
 When $n=1$, the {\em bottleneck distance}  $d_B$ is an easily computable extended metric on persistence diagrams \cite{KMN17}, defined as follows.
 
\begin{definition}[Bottleneck Distance] Let $M$, $N$ be two $1$-parameter persistence modules, with persistence diagrams $\dgm M$ and $\dgm N$.  
 
 A {\em matching} between $\dgm M $ and $\dgm N$ is a multi-bijection $\sigma: P \rightarrow Q$ between the points of a multi-subset $P$ in $\dgm M$ and a multi-subset $Q$ in $\dgm N$.
 
 The {\em cost of a matching $\sigma$}, denoted $c(\sigma)$, is the greatest among the costs of each matched pair of points, taken to be equal to the $\ell^\infty$ distance of the corresponding pair of points in $\RR\times \overline{\RR}$, with the convention that $\infty-\infty=0$ and $\infty-a=\infty$ for every $a\in \RR$, and  the costs of each unmatched point, taken to be equal to  the $\ell^\infty$ distance of that  point to the diagonal of $\RR^2$:
$$c(\sigma):=\max \left\{\max_{p\in P}\|p-\sigma(p)\|_\infty,\max_{p\notin P\coprod Q}\frac{p_2-p_1}{2}\right\} .$$

The {\em bottleneck distance} $d_B$ is defined as the smallest possible cost of any matching $\sigma$ between $\dgm M$ and $\dgm N$:
 $$d_B(\dgm M, \dgm N):= \min_{\substack{\sigma: P\to Q\\ P\subseteq \dgm M, Q\subseteq \dgm N}} c(\sigma). $$ 

\end{definition}

Observe that a matching $\sigma$ that pairs a point $(b,d)\in\RR^2$ with a point $(b',\infty) \in \RR\times \overline{\RR}$ has cost equal to infinity. Hence, the bottleneck distance between two persistence diagrams with a different number of points at infinity cannot be finite. On the contrary, matching points of the same type  always gives a finite (more convenient) cost that can be expressed as follows.

\begin{remark} By definition, the cost $c(\sigma)$ of $\sigma:P\to Q$ takes one of the two forms.
\begin{itemize}
\item If $c(\sigma)$ is realized by matching $p\in P$ to $q=\sigma(p)$, then $c(\sigma)=|s-t|$, where $s$ and $t$ are either both abscissas or both ordinates of $p$ and $q$, respectively; 
\item If $c(\sigma)$ is realized by some $p$ not in $P\coprod Q$, then $c(\sigma)=\frac{|s-t|}{2}$,  where $s$ and $t$ are the abscissa and the ordinate of $p$.
\end{itemize}

Briefly, we can say that $c(\sigma)$ is attained by some $\frac{|s-t|}{\delta}$,  with $\delta \in \{1,2\}$ and $s,t$ coordinates of points in $\dgm M\cup \dgm N$.
\end{remark}

When the number of parameters is $n\ge 2$, we can use the bottleneck distance to define an extended pseudo-metric between persistence modules $M$ and $N$  via restrictions to lines with positive slope.

The {\em restriction} of $M$ to a line of positive slope $L\subseteq \RR^n$  is the persistence module $M^L$ that assigns  $M_u$ to each $u\in L$, and whose transition maps $i_{M^L}(u,v)\colon (M^L)_u\to (M^L)_v$ for $u\preceq v\in L$ are the same as in $M$. Once a parameterization $u=\vec ms+b$ of $L$ is fixed, the persistence module $M^L$ is isomorphic to the $1$-parameter persistence module, by abuse of notation still denoted by $M^L$, that
\begin{itemize}
\item assigns to each $s\in\RR$ the vector space $(M^L)_s:=M_u$, and
\item whose transition map between $(M^L)_s$ and $(M^L)_t$ for $s\le t$ is equal to that of $M$ between $M_u$ and $M_v$ with $u=\vec ms+b$ and $v=\vec mt+b$.
\end{itemize}

\begin{definition}[Matching Distance] The {\em matching distance} between $M$ and $N$  is defined by
\[ 
\displaystyle d_{match}(M, N) := \displaystyle \sup_{L\colon u=s\vec m+b} \hat m^L \cdot d_B(\dgm M^{L}, \dgm N^{L}) 
\]
\noindent 
where $\hat m^L:= \min\{m_i\mid i=1,\ldots, n\}>0$ is the minimal coordinate of the direction vector $\vec m$ of $L$, and $L$ varies among all the lines of $\RR^n$ with positive slope, taken with standard normalization. 
\end{definition}

The role of the weight $\hat m^L$ in the definition of the matching distance is that of ensuring stability of persistence modules against perturbations of the filtrations that originate them (cf.~\cite{Cerri-et-al2013,Landi2018}). Such weight would vanish for lines parallel to the coordinate axes. Therefore, such lines are not considered in the definition of the matching distance. 

In the context of the matching distance between multi-parameter persistence modules, it is more convenient to include the factor $\hat m^L$ when computing the cost of a matching $\sigma^L$ between the points of a multi-subset in $\dgm M^{L}$ and the points of a multi-subset in $\dgm N^{L}$ for a fixed line $L$ with positive slope. Therefore, we introduce the following notation:
$${\rm cost}(\sigma^L) := \hat m^L c(\sigma^L) \ .$$

 Examples of computation of the matching distance in some simple cases are given in \Cref{sec:examples}. 

By definition, the matching distance is computed by taking the supremum of the bottleneck distance over all (infinitely many) lines with positive slope through the parameter space. However, in \cite{Bapat2022}, we have shown that this set of lines may be partitioned into a finite number of equivalence classes. For the two-parameter case (i.e., for $n = 2$), we will use these equivalence classes to generate a finite set of lines from which one may compute the matching distance. 

Henceforth in this paper, we always assume $n = 2$, which is the case for which we have  explicit results.  

We first review the notation and definitions (with associated figures) necessary to understand these equivalence classes following \cite{Bapat2022} for the two-parameter case. We summarize the relationship with analogous notions in the literature in  \Cref{sec:rel-work1}.

\begin{definition}[Positive Cone]\label{def:poscone}
For every point  $u$ in $\RR^2$, let $S_+(u)$ be the {\em positive cone with vertex $u$}: 
$$S_+(u):=\{v\in\RR^2:u\preceq v\} \ .$$ 

The boundary of the positive cone, $\partial S_+(u)$, decomposes into open faces.  In particular, $\partial S_+(u)$ can be partitioned by non-empty subsets $A$ of $\{1,2\}$ in the following way.  For $\emptyset\neq A\subseteq\{1,2\}$, define
\[
S_A(u):=\begin{cases}
\{u\} & \textrm { if } A=\{1,2\} \ ,\\
\{(x_1,x_2)\in\RR^2 \  | \ x_1=u_1, x_2>u_2 \}& \textrm{ if } A=\{1\} \ ,\\
\{ (x_1,x_2)\in\RR^2 \  | \ x_1>u_1, x_2=u_2 \} & \textrm{ if } A=\{2\} \ .\\
\end{cases}
\]

\begin{figure}[H]
\begin{center}
\begin{tikzpicture}
\input{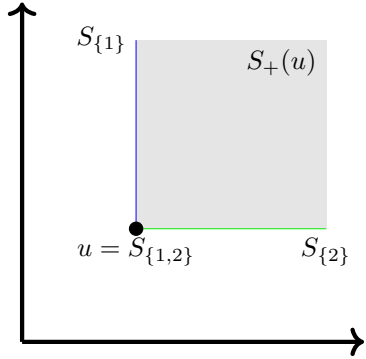}
\end{tikzpicture}
\caption{The positive cone $S_+(u)$ of $u\in\RR^2$ and the decomposition of its boundary into $S_{\{1\}}$, $S_{\{2\}}$ and $S_{\{1,2\}}$, which correspond respectively to the vertical boundary, horizontal boundary, and $u$.}\label{fig:positive_cone}
\end{center}
\end{figure}
\end{definition}

The following definition initially appears in \cite{Lesnick-Wright2015}.

\begin{definition}\label{def:push} A line $L$ with positive slope intersects $\partial S_+(u)$ in exactly one point, which we call the {\em push of $u$ onto $L$}, and denote by $\mathrm{push}_L(u)$:

$$\mathrm{push}_L(u):=L\cap \partial S_+(u)\ .$$ 

\begin{figure}[H]
\begin{center}
\begin{tikzpicture}
\input{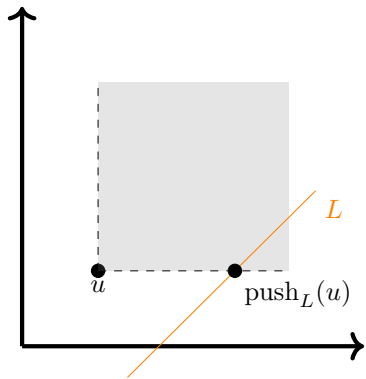}
\end{tikzpicture}
\caption{The push of $u$ along the line $L$. In this case, $A^L_u=\{2\}$, and so $u$ is to the left of $L$.}\label{fig:push}
\end{center}
\end{figure}
\end{definition}

Because of the partition of $\partial S_+(u)$ by  its open faces, there is a unique non-empty subset of $\{1,2\}$, denoted $A^L_u$, 
 such that $ \mathrm{push}_L(u)=L\cap  S_{A^L_u}(u)$.   Concretely, in the plane, ${A^L_u}=\{1\}$ means that $u$ is strictly to the right of $L$, ${A^L_u}=\{2\}$ means that $u$ is strictly to the left of $L$, ${A^L_u}=\{1,2 \}$ means that $u$  is on $L$.  More generally, when $\{1\}\subseteq{A^L_u}$, we say $u$ is to the right of $L$, and when $\{2\}\subseteq{A^L_u}$, we say $u$ is to the right of $L$, allowing for the fact that it may be true in either of these cases that $u$ is on $L$.
 
  We say that two lines $L,L'\subseteq \RR^2$ with positive slope have the {\em same reciprocal position} with respect to $u$ if and only if $A^L_u=A^{L'}_u$. 
  Given a non-empty subset $U$ of $\RR^2$, we write $L\sim _{U}L'$ if $L$ and $L'$ have the same reciprocal position with respect to $u$ for all $u \in U$. 
  Note that $\sim_U$ is an equivalence relation on lines with positive slope in $\RR^2$ \cite{Bapat2022}.

In this paper, it is necessary for us to extend this equivalence relation to include points in the projective completion of the real plane, $\mathbb{P}^2=\RR^2\cup\ell_\infty$,  with  $\ell_{\infty}$ the line at infinity of $\mathbb{P}^2$. Points on the line at infinity are given by homogeneous coordinates $[x_0:x_1:x_2]$ with $x_0=0.$  Note that a line with positive slope in $\RR^2$ will intersect $\ell_{\infty}$ at exactly one point $[0:v_1:v_2]$ with $\vec v= (v_1,v_2)\succ 0$   giving the direction of the line.

Given a point $v=[0:v_1:v_2] \in \ell_{\infty}$ with $v_1,v_2>0$, define
\[
S_A(v):=
\begin{cases}
\{v\} & \textrm { if } A=\{1,2\} \ ,\\
\{[0:x_1:x_2]\in \ell_{\infty} \  | \ x_1=v_1, x_2>v_2 \}& \textrm{ if } A=\{1\} \ ,\\
\{[0:x_1:x_2]\in \ell_{\infty} \  | \ x_1>v_1, x_2=v_2 \}& \textrm{ if } A=\{2\} \ .\\
\end{cases}
\]
For every such $v\in \mathbb{P}^2$ and line $L \subset \RR^2$ with positive slope, define $A^L_v$ to be the largest subset of $\{1,2\}$ such that $L \cap S_{A^L_v} \neq \emptyset$. As above, we say that two lines $L,L'$ in $\mathbb{R}^2$ with positive slope are in the {\em same reciprocal position} with respect to $v$ if and only if $A^L_v = A^{L'}_v$.

We then extend the equivalence relation on lines in $\RR^2$ with positive slope defined in \cite{Bapat2022} as follows: two lines $L$ and $L'$ are equivalent if and only if they are in the same reciprocal position with respect to a finite set $U$ of points in $\mathbb{P}^2$: 
\[
L \sim_U L' \text{ if and only if } A^L_u = A^{L'}_u  \;  \forall u \in U \ .
\]

We introduce the following notation for simplicity.
For a line  $L\subset \RR^2$  parameterized as $\vec{m}s + b$, and for $u \in \mathbb{R}^2$, we set $p_L(u)$ to be the unique real number such that 
\[\push_L(u) = \vec{m}\cdot p_L(u) + b.\]
In other words, $p_L(u)$ is the parameter value of the push of $u$ to $L$. 
We note that $p_L(u)$ depends on the chosen parameterization of the line $L$.
However, we implicitly assume that the parameterization we choose is the standardly normalized one.

The formula for the quantity $p_L(u)$ depends on whether $u$ lies to the left (i.e. $A^L_u=\{2\}$) or to the right (i.e. $A^L_u=\{1\}$) of the line $L$.
If $u$ lies to the left of $L$, then 
\begin{equation}\label{eq:param-1}
p_L(u) = \frac{u_2}{m_2} - \frac{b_2}{m_2}.
\end{equation}
If $u$ lies to the right of $L$, then 
\begin{equation}\label{eq:param-2}
p_L(u) = \frac{u_1}{m_1} - \frac{b_1}{m_1}.
\end{equation}
If $u$ lies on $L$, then 
\begin{equation}\label{eq:param-3}
p_L(u) = \frac{u_1}{m_1} - \frac{b_1}{m_1}= \frac{u_2}{m_2} - \frac{b_2}{m_2}.
\end{equation}

Note that, thanks to the standard parametrization of a line $L$, pushing  any two points $u,v\in \RR^2$ onto $L$ yields parameter values $p_L(u), p_L(v)\in \RR$ such that 
\begin{equation}\label{eq:param-diff}
 p_L(u)-p_L(v)=\left\{\begin{array}{ll}
\push_L(u)_2-\push_L(v)_2  & \mbox{ if $L$ has slope greater than 1,}\\
 \push_L(u)_1-\push_L(v)_1  & \mbox{ otherwise.} 
\end{array}\right. 
\end{equation}

For sufficiently nice modules (such as those that satisfy Property (P) from \Cref{sec:workingassumption}), there is a finite set $C$ of points in $\RR^n$ with the property that, for a given line $L$, the set $\left\{\push_L(c)|c\in C\right\}$
can be used to determine the coordinates of the points in the persistence diagram associated to $L$.
(See, for example, 
\cite[Thm. 2]{Bapat2022})

Let $L$ and $L'$ be lines in $\RR^2$ with positive slope parameterized by $L\colon u =\vec m s+b$ and $L'\colon u' =\vec m ' s'+b'$ respectively.  Let $S$ be the set of all matchings between $M^L$ and $N^L$, and let $S'$ be the set of all matchings between $M^{L'}$ and $N^{L'}$.

\begin{lemma}\label{lem:Gamma}  If $L\sim_{\closure C_M\cup \closure C_N}L'$, then there exists a bijection $\Gamma_{L,L'}$ between $S$ and $S'$.
\end{lemma}

\begin{proof}
  Since $L\sim_{\closure C_M\cup \closure C_N}L'$, we know that $L\sim_{\overline{C}_M}L'$ and $L\sim_{\overline{C}_N}L'$.  Therefore, by our previous paper \cite{Bapat2022}, there exist multi-bijections $\gamma_M\colon\dgm M^L\to\dgm M^{L'}$ and $\gamma_N\colon\dgm N^L\to\dgm N^{L'}$.   For any matching $\sigma \colon P\to Q$ in $S$ such that $P \subseteq\dgm M^L$ and $Q\subseteq N^L$, set  $\Gamma_{L, L'}(\sigma)=\gamma_N\circ\sigma\circ(\gamma_M)^{-1}$.
  In other words, the following diagram commutes:

\[\begin{tikzcd}
\dgm M^L \arrow{r}{\sigma} \arrow{d}{\gamma_M} & \dgm N^L \arrow{d}{\gamma_N} \\
\dgm M^{L'} \arrow{r}{\Gamma_{L,L'}(\sigma)} &\dgm N^{L'}
\end{tikzcd}\]

Since $\gamma_M$ and $\gamma_N$ are bijections, $\Gamma_{L,L'}(\sigma)$ is easily seen to be a matching of $S'$. Similarly, 
we could have just as easily taken a matching $\tau\in S'$, and created a matching $(\gamma_N)^{-1}\circ\tau\circ\gamma_M$ in $S$.  Therefore, $\Gamma_{L,L'}$ establishes a bijection between the sets $S$ and $S'$.
\end{proof}


More explicitly, the multi-bijection $\gamma_M$ behaves as follows.

Let $u,v \in \overline{C}_M$ be such that the pair $(p_{L'}(u),p_{L'}(v))\in\dgm M^{L'}$. Then $\gamma_M$ gives a bijection between the parametrized values of pairs $u,v\in\overline{C}_M$ on $L'$ and $L$, i.e.,  
 $$\gamma_M^{-1}(p_{L'}(u),p_{L'}(v))=(p_{L}(u),p_{L}(v))\in\dgm M^L.$$  

Under the assumptions above, since $L\sim_{\overline{C}_M}L'$, both lines intersect the same face of the positive cone of $u$.
This means that there exists some $i\in\{1,2\}$ such that $(\vec m \cdot p_{L}(u)+ b)_i=(\vec m'\cdot p_{L'}(u)'+b')_i=u_i$.  This enables us to solve for $p_{L'}(u)$ from $p_{L}(u)$:

\begin{align}
(\vec m \cdot p_{L}(u)+b)_i=(\vec m' \cdot p_{L'}(u)+b')_i &\Leftrightarrow \nonumber\\
m_i p_{L}(u)+b_i=m_i' p_{L'}(u)+b'_i  &  \Leftrightarrow \nonumber\\
p_{L'}(u)=\frac{m_i}{m_i' }p_{L}(u)+\frac{b_i-b_i'}{m_i'}  \ .  & \label{eq:s'}
\end{align}

Similarly, there is some $j\in\{1,2\}$ such that
\begin{equation*}
p_{L'}(v)=\frac{m_j}{m_j' }p_{L}(v)+\frac{b_j-b_j'}{m_j'}.
\end{equation*}

Analogously, if $u \in \overline{C}_M$ is such that the pair $(p_{L'}(u),\infty)\in\dgm M^{L'}$, we can solve for $p_{L'}(u)$ from $p_{L}(u)$ as in \eqref{eq:s'}.

\subsection{Connections to related works}\label{sec:rel-work1}
As mentioned in \Cref{sec:introduction}, there are other papers that look to understand the matching distance for bi-persistence modules.  Our approach to computing the matching distance in this paper is closely related to two groups of papers:
\begin{itemize}
    \item The papers \cite{Kerber-Lesnick-Oudot2018,Lesnick-Wright2015} present an approach in the dual space of lines, and 
    \item The papers \cite{Cerri2019geometrical, Cerri2020,ethier2023geometry} present an approach in topological persistence.
\end{itemize}
Where appropriate throughout the paper, we will highlight the similarities and differences between our paper and these two groups (termed the ``dual approach" and the ``topological approach", respectively).  The first such section here will provide connections and analogues to the dual and topological approach for the content in \Cref{sec:notation}.

Firstly, note that \Cref{def:poscone} is related to the construction of the\emph{ extended Pareto grid} of topological persistence from \cite{Cerri2019geometrical}.  One of the main differences between this paper and \cite{Cerri2019geometrical} is the fact that we assume we have a finite number of critical parameter values which determine the persistence diagrams along lines, while the persistence diagram along a line $L$ in \cite{Cerri2019geometrical} is determined by the intersection of $L$ with the extended Pareto grid (see Theorem 2 in \cite{Cerri2019geometrical}).

Secondly, we note that the bijection $\Gamma_{L,L'}$ from \Cref{lem:Gamma} corresponds to the transport of matchings introduced in \cite{Cerri2019geometrical} (Subsection 4.1).
Specifically, the map $\Gamma_{L,L'}$ associates the matching corresponding to the
line $L$ with the matching corresponding to the line $L'$, constructed by
tracking the movement of points in the persistence diagrams as $L$ transitions to $L'$. A similar concept also appears in \cite{Kerber-Lesnick-Oudot2018}, in the proof of Theorem 4.6.  In all three papers (\cite{Kerber-Lesnick-Oudot2018,Cerri2019geometrical} and this paper) it is necessary to identify a bijection between points in the persistence diagrams along $L$ and $L'$, in order to bound the weighted cost of the bottleneck distance on some equivalence class of lines.  This bound and the finiteness of the number of such equivalence classes is what allows us to obtain a maximum for the matching distance.

\subsection{Working assumptions}\label{sec:workingassumption}

Our goal in this paper is to identify a finite set of lines with positive slope from which to compute the matching distance between two finitely presented persistence modules, $M$ and $N$. Our strategy is that of partitioning the set of lines with positive slope  into equivalence classes where it is easy to understand how the cost of a matching between persistence diagrams along lines changes when moving such lines around, for example by translations and rotations.  

To this end, we start with requesting that the following property holds for all the $n$-parameter persistence modules $M$  considered hereafter.

\begin{property*}\label{property:p}[{\bf P}]
There exists a finite set  $C_M$ of points in $\RR^n$, called {\em critical values}, for which the ranks of transition maps in the $n$-parameter persistence module $M$ are completely determined by the ranks of transition maps between values in $\closure C_M$, the closure of $C_M$ under least upper bound.  More specifically, for any $u,v\in\RR^n$,
\[\rank i_{M}(u,v)=\rank i_{M} (\bar{u},\bar{v})\]
   with 
   \[\closure{u} = \max \{u' \in \closure C_M| u'\preceq u\},\]
   \[\closure{v} = \max \{v' \in \closure C_M | v'\preceq v\}\] 
if $\{u'\in \closure C_M| u'\preceq u\}$ is non-empty, and $\rank i_{M}(u,v)=0$ otherwise. 
Note that $C_M$ is empty if and only if $M$ is trivial.
\end{property*}

In the case of a finitely presented persistence module $M$, one such set $C_M$ exists and is given by the set of grades of its generators and relators.  
More concretely, Property (P) is satisfied by modules $M$  obtained by taking the persistent homology over a filed $\F$ of simplicial complexes $K$ filtered by a tame and one-critical filtration $\KK=\{K^u\}_{u\in\RR^n}$: $M = \h_q(\KK; \F)$. Indeed, fixing a discrete gradient vector field $V$ consistent with the filtration $\KK$ and defining $C_M$ to be the set of entrance values of simplices $\sigma$  that are critical in the gradient vector field $V$,  $M$ satisfies Property (P) with set of critical values given by $C_M$ as proved in \cite[Thm. 1]{Bapat2022}.

If both $M$ and $N$ have Property (P), then we may partition the set of lines with positive slope in $\mathbb{R}^n$ via the equivalence relation $\sim_{\closure C_M\cup \closure C_N}$. For any two lines $L\sim_{\closure C_M\cup \closure C_N}L'$, the persistence diagrams $\dgm M^L$ and $\dgm M^{L'}$ (respectively $\dgm N^L$ and $\dgm N^{L'}$) are in bijection \cite{Bapat2022}[Thm. 2].  
 
 As we will see in \Cref{lem:Gamma} later, these bijections will extend to a bijection, denoted $\Gamma_{L,L'}$,  between the sets of matchings that are used to compute both $d_B(M^L,N^L)$ and $d_B(M^{L'},N^{L'})$ .  Moreover, the bijection $\Gamma_{L,L'}$ can be used to compute the cost of a matching for line $L'$ from the cost of the corresponding matching for line $L$, provided the lines are in the same equivalence class. 

Ideally, one would like to use $\Gamma_{L,L'}$ to show that a line $L$ on which $d_B(M^L,N^L)$ attains a maximum is a line which passes through two points in $\closure C_M\cup \closure C_N.$  However, this is not the case, as \Cref{ex:need-diag,ex:need-omega} show.  The reason is that, within an $\sim_{\closure C_M\cup \closure C_N}$-equivalence class, it is possible  that there is no ``winning'' matched pair of points or ``winning'' matching that works for all lines within that equivalence class, i.e., the winning matching may ``switch" within the same $\sim_{\closure C_M\cup \closure C_N}$ equivalence class.

Therefore, the equivalence relation $\sim_{\closure C_M\cup \closure C_N}$ must be refined, in order to generate equivalence classes in which there exists at least one matching and matched pair which computes the bottleneck distance for all lines in a given equivalence class.  This requires the addition of {\em switch points}: points $\omega$ in the projective completion for which the cost of matching some pair $u,v$ equals the cost of matching some other pair $w,x$, with $u,v,w,x\in\closure C_M\cup \closure C_N$ for any line through $\omega$.

 So, with the goal of computing the matching distance using the method described above in the case $n=2$, we introduce another property, called (Q),  that critical values $C_M$ and $C_N$ of  2-parameter persistence modules $M$ and $N$ are required to satisfy, involving a set of switch points.

\begin{property*}[{\bf Q}]\label{propq}
There exists a finite set of points $\Omega = \Omega(C_M,C_N)$ in $\mathbb{P}^2$ such that for all $u, v, w, x \in \closure C_M\cup \closure C_N$, not necessarily distinct, and $\delta, \eta \in \{1,2\}$, also not necessarily distinct, the sign (positive, negative or null) of 
\[\Delta_L(u,v,w,x;\delta,\eta):=\frac{|p_L(u) - p_L(v)|}{\delta} - \frac{|p_L(w) - p_L(x)|}{\eta},\]
remains the same for all lines $L$ in the same equivalence class with respect to $\closure C_M\cup \closure C_N \cup \Omega$.
\end{property*}

 Once the equivalence relation $\sim_{\closure C_M\cup \closure C_N}$ is refined to $\sim_{\closure C_M\cup \closure C_N\cup{ \Omega}}$, a winning matching $\sigma$ and a winning matched pair $u,v\in \closure C_M\cup \closure C_N$ can be determined for each equivalence class. Indeed,  Property (Q) implies the following Property (Q'), as \Cref{prop:Q} will prove.

\begin{property*}[\bf Q']\label{propq'}
There exists a finite set of points $\Omega = \Omega(C_M,C_N)$ in $\mathbb{P}^2$ with the following property.
Let $L$ be any line, and let $\sigma$ be an optimal matching  with respect to computing the bottleneck distance between $\dgm M^L$ and $\dgm N^L$.
Suppose that the cost $c(\sigma)$ of the matching $\sigma$ is realised by the pushes of $u,v \in \closure C_M\cup \closure C_N$ onto $L$.
Then, if $L'$ is any other line in the same equivalence class as $L$ with respect to $\closure{C}_M\cup\closure{C}_N \cup \Omega$, the matching $\Gamma_{L,L'}(\sigma)$ is an optimal matching with respect to computing the bottleneck distance between $\dgm M^{L'}$ and $\dgm N^{L'}$.
Moreover, its cost $c(\Gamma_{L,L'}(\sigma))$ is realised by the pushes of $u$ and $v$ onto $L'$.
In other words,
\[c(\sigma) = \frac{|p_L(u) - p_L(v)|}{\delta} \iff c(\Gamma_{L,L'}(\sigma)) = \frac{|p_{L'}(u) - p_{L'}(v)|}{\delta}.\]
\end{property*}


\begin{proposition}\label{prop:Q}
 Property (Q)  implies Property (Q'). 
\end{proposition}

\begin{proof}
Choose a line $L$ and let $\sigma$ be a matching between $\dgm M^L$ and $\dgm N^L$ that has cost
\[c(\sigma) = \frac{|p_L(u) - p_L(v)|}{\delta_\sigma}\]
for some $u, v \in \closure C_M\cup \closure C_N$ and $\delta_\sigma \in \{1,2\}$.

Let $L'$ be a line that is equivalent to $L$ with respect to $\closure C_M\cup \closure C_N \cup \Omega$.
Property (Q) implies that 
\[c(\Gamma_{L,L'}(\sigma)) = \frac{|p_L(u) - p_L(v)|}{\delta_\sigma}.\]
This tells us that within an equivalence class under $\sim_{\closure C_M\cup \closure C_N \cup \Omega}$, if the pair $u,v \in \closure C_M\cup \closure C_N$ obtains the cost of $\sigma$, then that pair also obtains the cost of $\Gamma_{L,L'}(\sigma)$, although it may not be the unique such pair to do so.

Now suppose that $\tau$ is the optimal matching with respect to computing the bottleneck distance between $\dgm M^L$ and $\dgm N^L$. Let $u_\tau, v_\tau \in \closure C_M\cup \closure C_N$ and $\delta_\tau \in \{1,2\}$ such that  
\[c(\tau) = \frac{|p_L(u_\tau) - p_L(v_\tau)|}{\delta_\tau}.\]
Since $\tau$ is the optimal matching, for any other matching $\sigma$ between $\dgm M^L$ and $\dgm N^L$, 
it is true that $c(\tau) \leq c(\sigma)$.
Therefore, again by Property $(Q)$, we have 
\[c(\Gamma_{L,L'}(\tau)) \leq c(\Gamma_{L,L'}(\sigma)).\]
Therefore $\Gamma_{L,L'}(\tau)$ is a (possibly not unique) matching that obtains the bottleneck distance between $\dgm M^{L'}$ and $\dgm N^{L'}$ , and $u_\tau, v_\tau \in \closure C_M\cup \closure C_N$ is a (possibly not unique) pair of critical values which obtains the cost of $\Gamma_{L,L'}(\tau)$.
\end{proof}

Finally, because \Cref{ex:need-diag} shows that the diagonal direction is also needed, we set $\overline{\Omega}$ to be the union of $\Omega$ with the singleton $\{[0:1:1]\}$. This will allow for the computation of the matching distance between $M$ and $N$ from lines through two points in $\closure C_M\cup \closure C_N\cup \overline{\Omega}$ via \Cref{thm:main}, assuming that $M$ and $N$  satisfy properties (P) and  (Q).  

As mentioned above, all finitely presented persistence modules satisfy Property (P). In \Cref{sec:switch}, we provide the theoretical justification also for the existence of the set $\Omega$ ensuring Property (Q).

\subsection{Examples}\label{sec:examples}

We now give some simple examples of computation of the matching distance in the case of two parameters. The considered persistence modules will be direct sums of rectangle modules, i.e. persistence module that are trivial outside a supporting rectangle $[u_1,v_1)\times [u_2,v_2)$, with $u=(u_1,u_2)$, $v=(v_1,v_2)\in\RR^2$, and equal to $\F$ with all the transition maps equal the identity on the supporting rectangle with bottom left corner $u$ and a top right corner $v$.

\Cref{ex:need-diag} shows that it does not suffice to consider only lines though pairs of critical values of the modules. 

\begin{example}\label{ex:need-diag}
Let $M$ be the rectangle module with support $[2,\infty) \times [2,7)$, and let $N$ be the rectangle module with support $[2,\infty) \times [2,10)$. 

\begin{figure}[h!]
    \centering
    \includegraphics[width=50mm,scale=0.7]{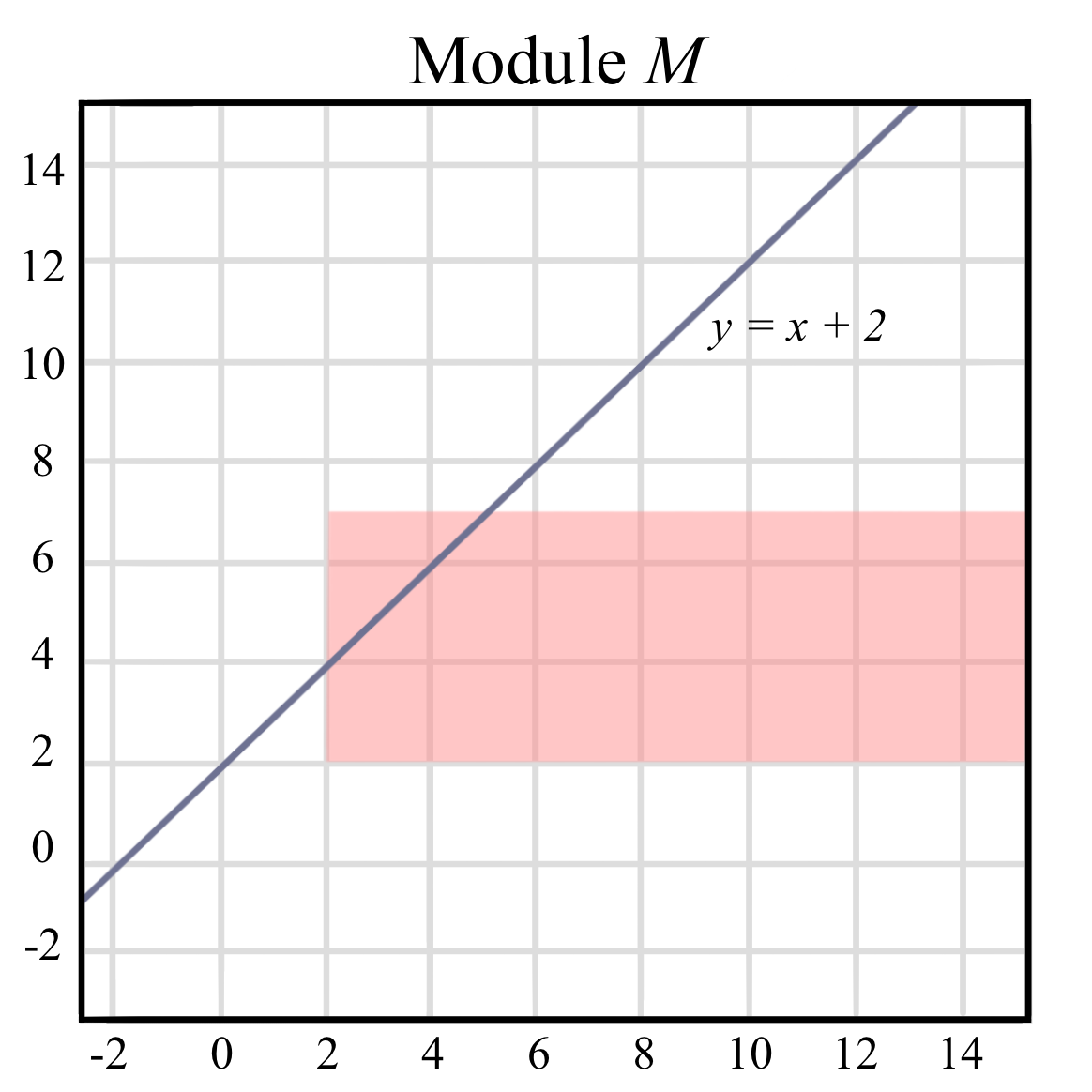}
    \includegraphics[width=50mm,scale=0.7]{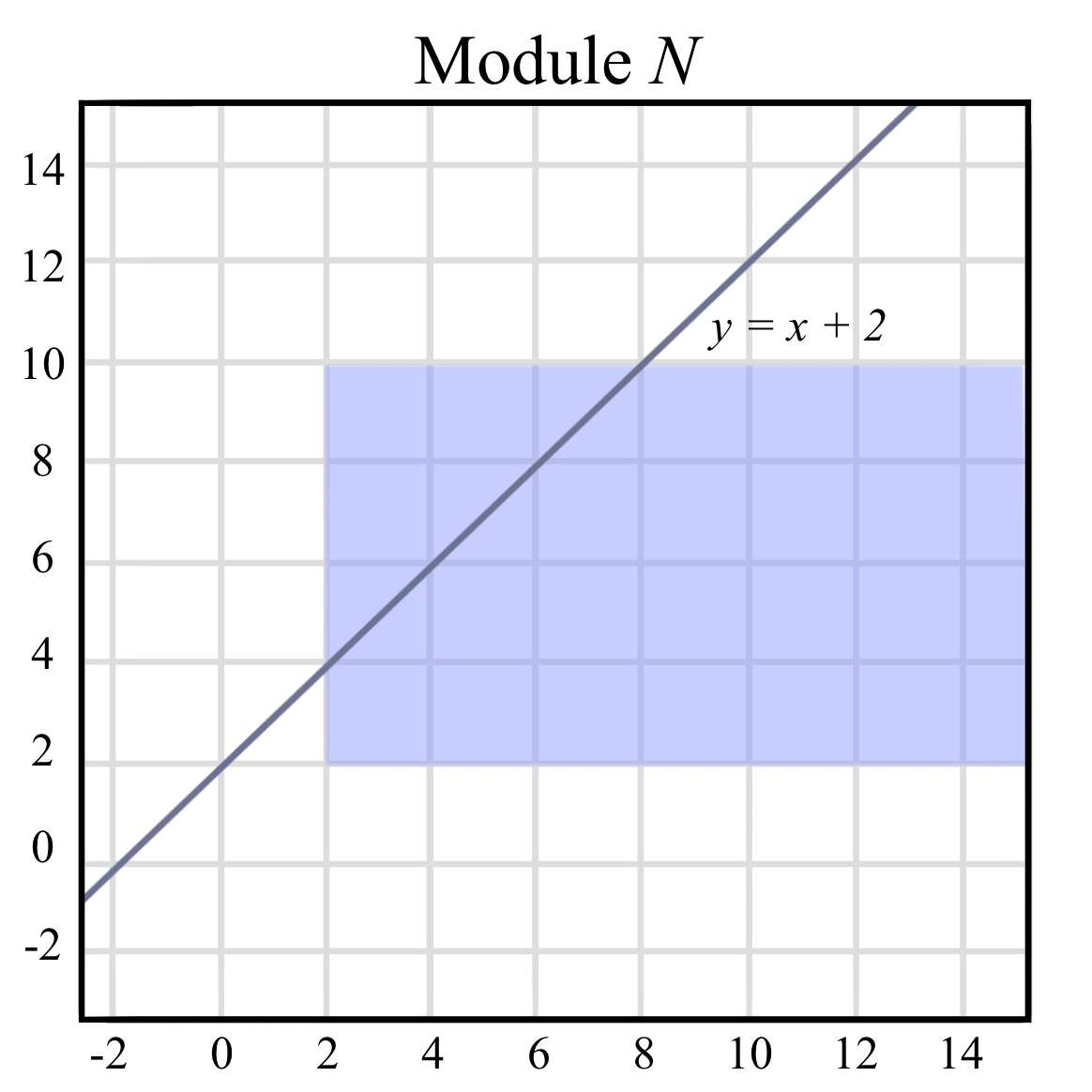}
    \caption{The persistence module considered in \Cref{ex:need-diag}: The matching distance is not attained at any line through two critical values.}
    \label{fig:need-diag}
\end{figure}

For lines $L$ with equation $y=mx + q$ with $m \leq 1$ and $q \leq 7 - 2m$, we get weighted bottleneck distance $3$, thus achieving the matching distance, because the weighted distance is smaller than $3$ for all other lines. In particular, the line through $(2,2)$, point in $C_M \cap C_N$, with diagonal direction achieves the matching distance between $M$ and $N$. 

Symmetrically, if we consider the rectangle persistence modules $M =[2,7)\times [2,\infty) $ and $N = [2,10) \times [2,\infty)$, the situation is reversed, that is, the matching distance is realised by lines of slope greater than or equal to $1$. 

In neither this example nor its symmetric version is it possible to achieve the matching distance by a line with positive slope through two points in $C_M \cup C_N$ or their closure with respect to the least upper bound.  However, in both cases, the matching distance is achieved on a line that passes through one point in $C_M \cup C_N$, with diagonal slope. 
\end{example}

\Cref{ex:diag-not-suff} shows that is does not suffice to consider only diagonal lines. (For a further discussion on the role that diagonal lines play in computing switch points and the matching distance, see \Cref{rem:diaglines}).

\begin{example}\label{ex:diag-not-suff} 
Let $M$ be the rectangle decomposable module given as the direct sum of two rectangle modules, one with support $[0,7) \times [0,7)$, the other one with support $[0,7) \times [4,11)$. Let $N$ be the rectangle decomposable module given as the direct sum of the rectangle modules with support $[0,7) \times [0,11)$ and $[0,7) \times [4,7)$. 

For diagonal lines $L$ with equation $y=x + q$ for $q \leq 0$ or $ q \geq 4$, we have that $M^L = N^L$ so that $d_{B}(M^L, N^L) = 0$. For $0 < q < 4$, a straightforward computation shows that the maximum bottleneck distance achieved is for $q=2$, where the cost is $2$. However, for the line $L'$ through the points $(0,0)$ and $(7,11)$, the bottleneck distance is $d_{B}(M^L, N^L) = 4$ and $\hat m^{L'} = \frac{7}{11}$, so that the matching distance between $M$ and $N$ is at least $\frac{28}{11}$. This example shows that diagonal lines do not suffice.

\begin{figure}[h!]
    \centering
    \includegraphics[width=50mm,scale=0.7]{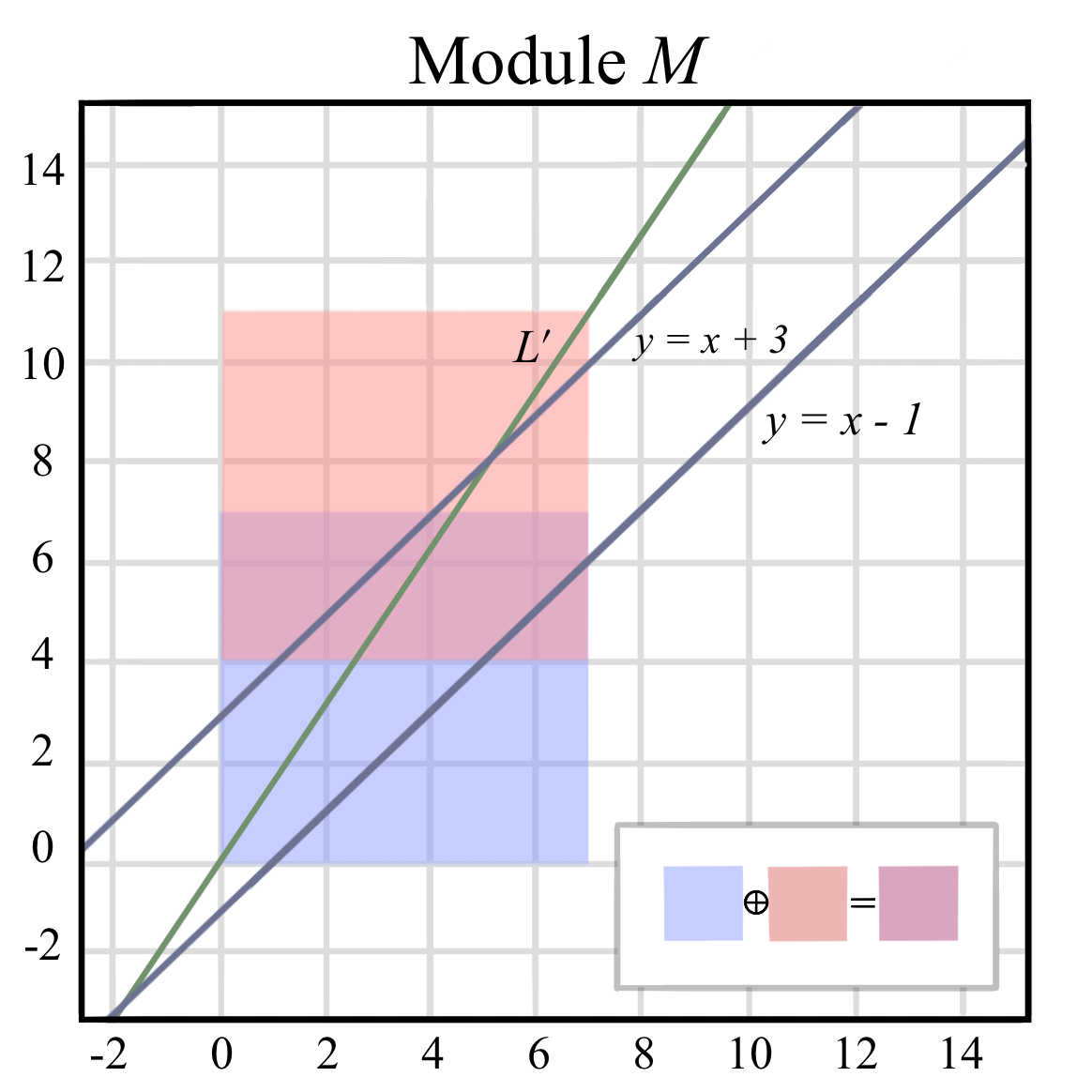}
    \includegraphics[width=50mm,scale=0.7]{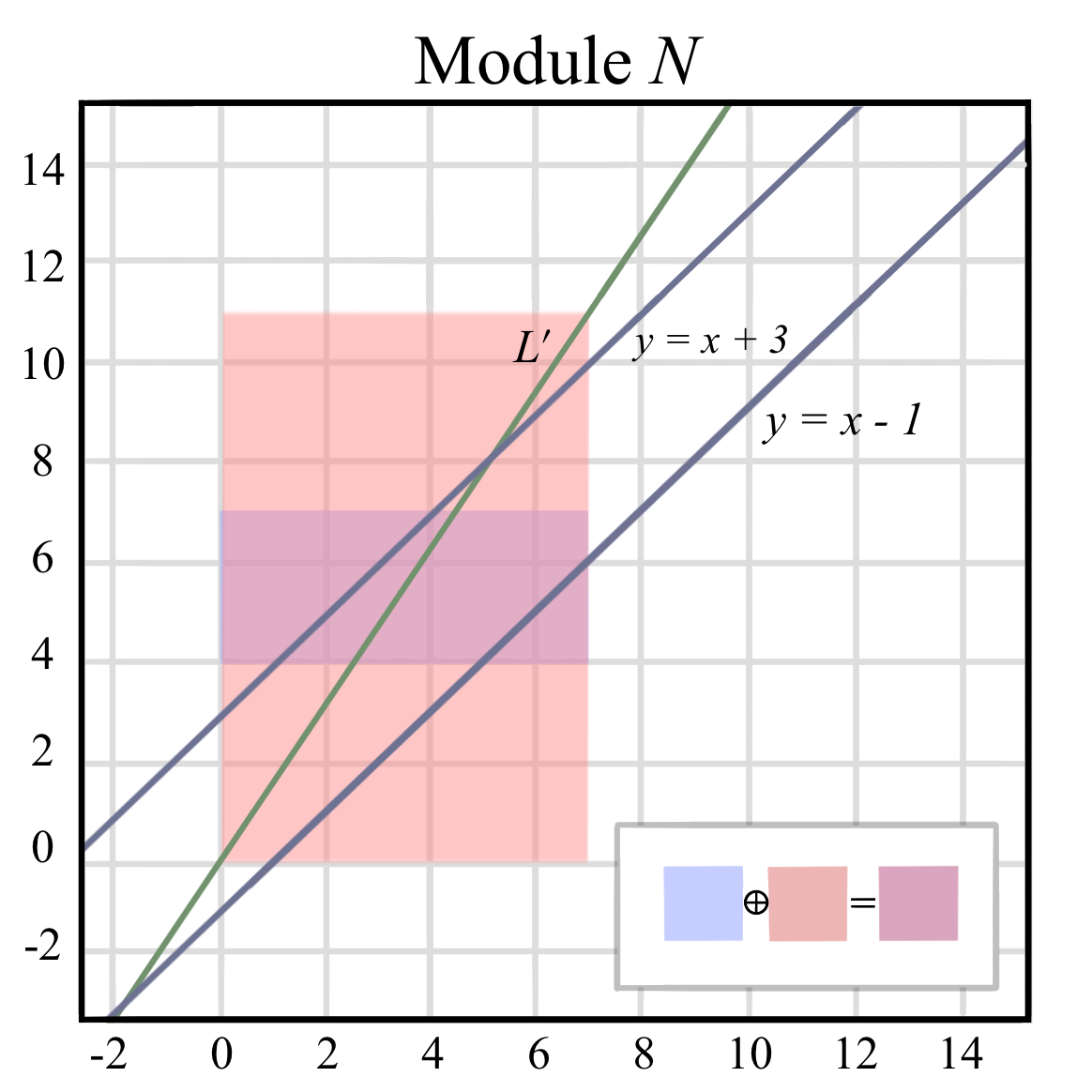}

    \caption{The persistence modules considered in \Cref{ex:diag-not-suff}: The matching distance is not attained at any diagonal line. }
  \label{fig:diag-not-suff}
\end{figure}
\end{example}

The next example, illustrated in \Cref{fig:need-omega1,fig:need-omega2}, shows that not even considering both lines through critical values or their least upper bounds and diagonal lines ensures we achieve the matching distance. This fact  will motivate the introduction of a further set of points, called switch points, in \Cref{sec:switch}.

\begin{example}\label{ex:need-omega} 

Let $M$ be the rectangle decomposable module which is the direct sum of two rectangle modules, one with underlying rectangle $[0,7) \times [0,8)$, the other one with underlying rectangle $[0,7) \times [4,11)$. 
Let $N$ be the rectangle decomposable module sum of two rectangle modules  whose underlying rectangles are $[0,7) \times [4,8)$ and $[0,7) \times [0,11)$. 

\begin{figure}[h!]
    \centering
    \includegraphics[width=50mm,scale=0.7]{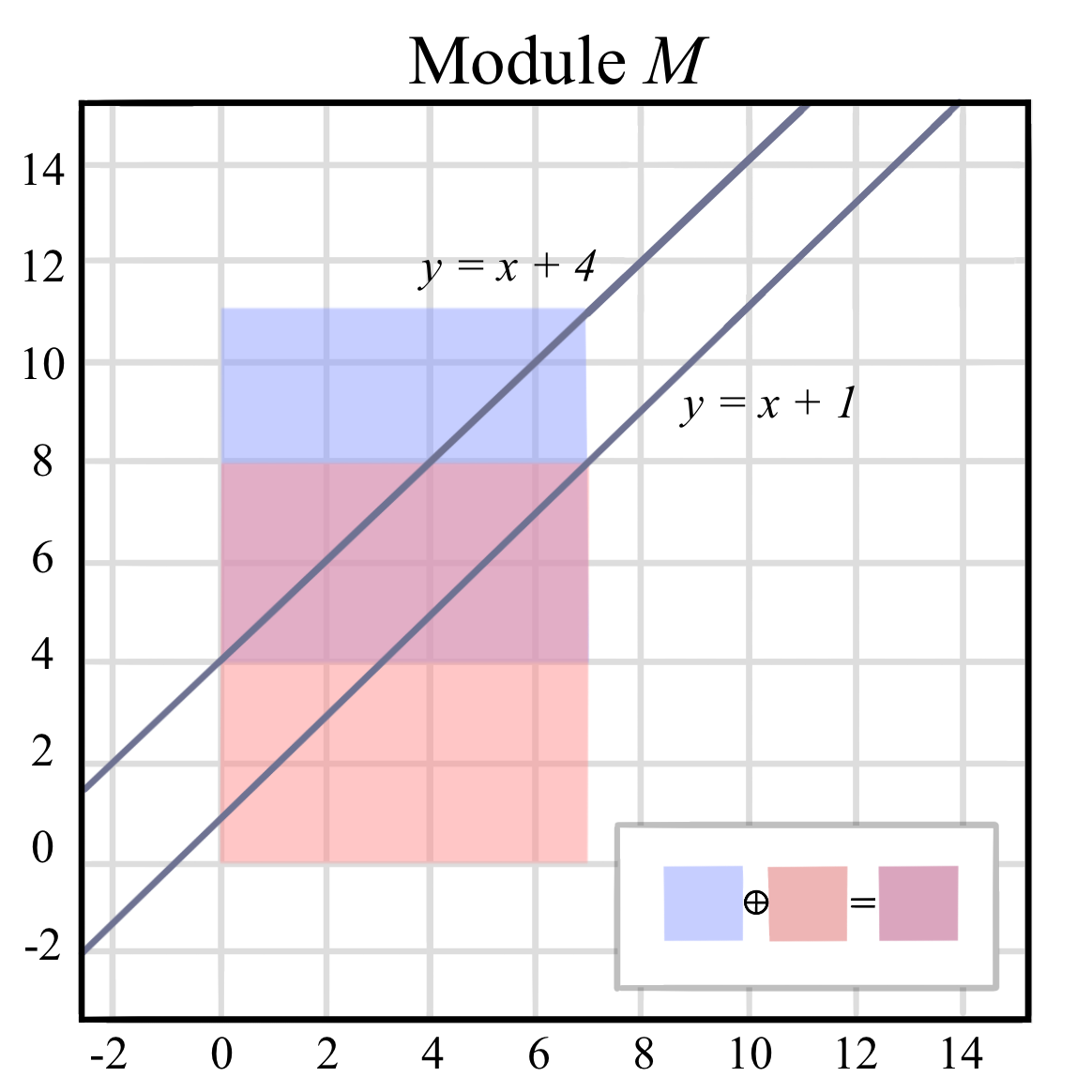}
    \includegraphics[width=50mm,scale=0.7]{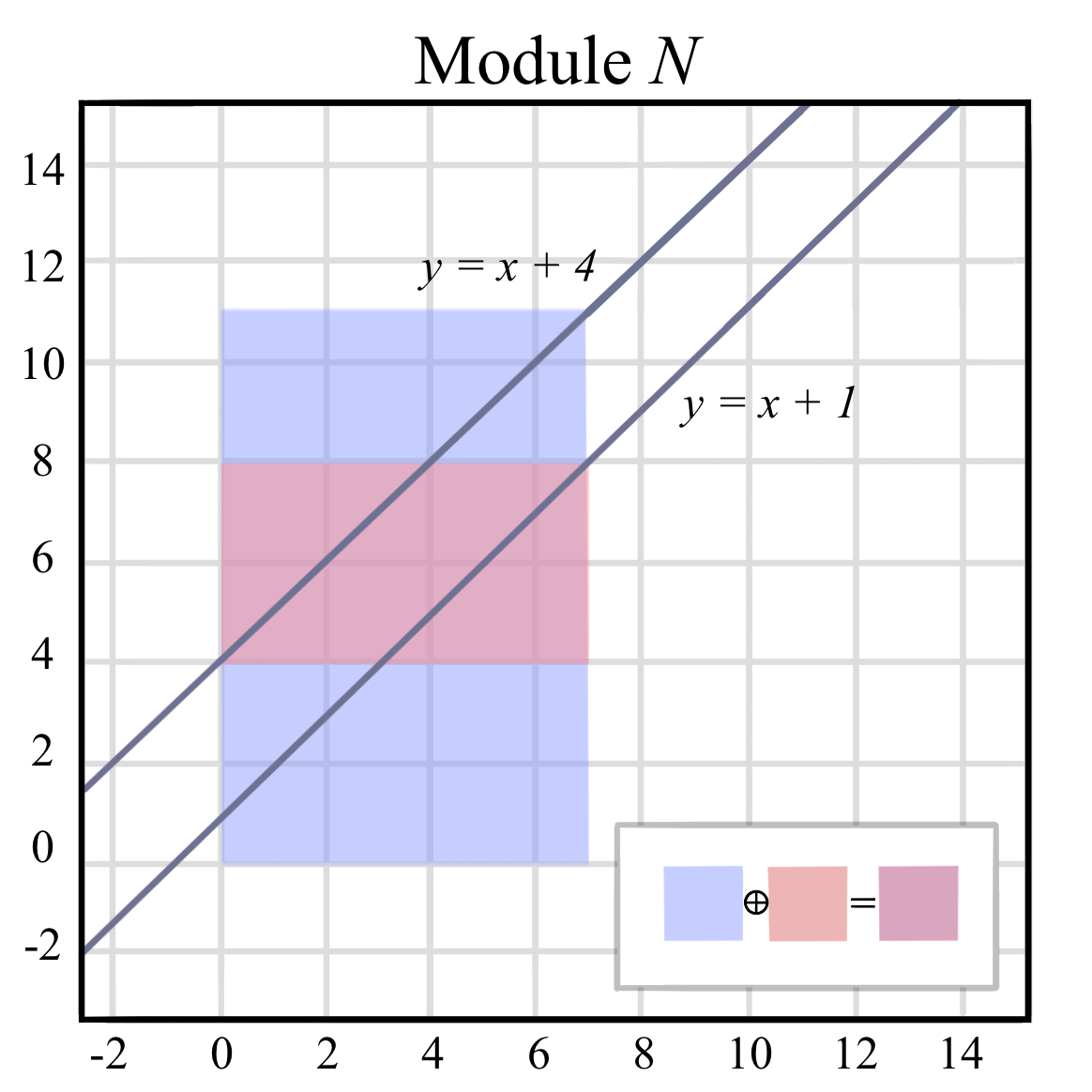}
    \caption{The persistence modules considered in \Cref{ex:need-omega}:  Neither lines through critical values nor diagonal lines suffice to achieve the matching distance.}
    \label{fig:need-omega1}
\end{figure}

A grid search on the line parameters $\theta$ and $b$ suggests that the matching distance is achieved neither along diagonal lines nor along lines through two points in the closure with respect to lowest upper bound of $C_M \cup C_N$, as we find a larger weighted bottleneck distance along a line that is neither diagonal nor through two critical points than along any diagonal or critical-point induced line in the grid search (see \Cref{fig:need-omega2}). 

\begin{figure}
    \centering
    \includegraphics[scale = 0.6]{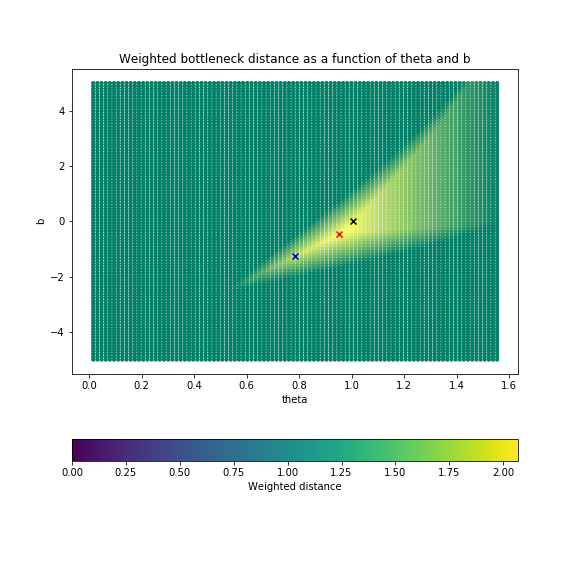}
    \caption{Weighted bottleneck distance for the modules of \Cref{ex:need-omega} across the line parameters. On the horizontal axis we have the values of the angle $\theta$, giving the slope of the line. On the vertical axis we have the parameter $b$ giving the origin of the line. The red cross corresponds to the maximum value over all lines computed in the approximation, the black cross the maximum value for lines through pairs of critical values and the blue cross is the maximum value for diagonal lines. This particular example shows that it is not sufficient to consider only lines through pairs of critical values or diagonal lines. This is because the matching distance must be at least as much as the value computed at the red cross.}

    \label{fig:need-omega2}
\end{figure}

The following considerations confirms this: 
for diagonal lines $L$ with equation $y=x + q$ for $q \leq 1$ or $ q \geq 4$, we have that $M^L = N^L$ so that $d_{B}(M^L, N^L) = 0$. For $1 < q < 4$, straightforward computation shows that the maximum bottleneck distance achieved is for $q=2.5$, where the cost is $\frac{3}{2}$. However, for the line $L'$ through the points $(\frac{7}{2},6)$ and $(7,11)$, the bottleneck distance is $d_{B}(M^L, N^L) = 3$ and $\hat m^{L'} = \frac{7}{10}$, so that the matching distance between $M$ and $N$ is at least $\frac{21}{10}$. Out of the lines through two points in the closure with respect to lowest upper bound of $C_M \cup C_N$, the only line that does not give $M^L = N^L$ is the line through $(0,0)$ and $(7,11)$, which has weighted bottleneck distance $\frac{21}{11}$. This example shows that it does not suffice to consider only diagonal lines and lines through two points in the closure with respect to the lowest upper bound of $C_M \cup C_N$. 
\end{example}


\section{Main result}\label{sec:mainresult}
The goal of this section is to prove the following main result.

\begin{restatable}{theorem}{main}\label{thm:main}
Let $M,N$ be two $2$-parameter persistence modules. If $M,N$ are both trivial, then $d_{match}(M,N)=0$. Otherwise, assume that $M$ and $N$  satisfy Properties (P) and (Q). Let $C_M$ and $C_N$ be finite sets of critical values in $\RR^2$ as ensured by Property (P), and let $\Omega(C_M,C_N)$ be a finite set of points in $\mathbb{P}^2$ as ensured by Property (Q). Then the matching distance between $M$ and $N$ is achieved on a line through two distinct points, one in $\closure C_M\cup \closure C_N \cup \Omega_P$ and the other one in $\closure C_M\cup \closure C_N \cup \closure\Omega$, with 
$\Omega_P=\Omega(C_M,C_N)\cap  \RR^2$, and $\overline{\Omega}=\Omega(C_M,C_N)\cup \{[0:1:1]\}$.
\end{restatable}

In particular, since finitely presented persistence modules satisfy properties (P) and (Q), \Cref{thm:main} holds for such persistence modules.  We note here that this is a statement of existence, not uniqueness:  it is possible that the matching distance is achieved on many other lines in addition to the one guaranteed by \Cref{thm:main}.

We prepare the proof of this theorem (given in \Cref{sec:proof-main}) by showing some preliminary claims.  In \Cref{sec:claims}, \Cref{lem:translations} states that, for a line with positive slope, there is a direction of translation within its $\sim_{\closure C_M\cup \closure C_N \cup \closure\Omega}$ equivalence class which does not decrease the cost of matchings.  \Cref{lem:rotations} states a similar result for rotations about a point of the line.  \Cref{sec:VandHlines} discusses the role that vertical and horizontal lines play in computing the matching distance. In particular, they are not needed in the set of lines we use to compute the matching distance.

\subsection{Preliminary claims}\label{sec:claims}

In this section, we fix  a line $L$ with positive slope, we let $P \subseteq \dgm M^L$ and $Q \subseteq \dgm N^L$, and we take $\sigma \colon P \rightarrow Q$ to be a matching such that the pair $u,v\in\overline{C}_M\cup\overline{C}_N$ achieves the cost of $\sigma$:
${\rm cost}(\sigma) = \hat m^L\frac{|p_L(u)-p_L(v)|}{\delta},$
with 
\[
\delta = 
\begin{cases}
     1,& \text{if ${\rm cost}(\sigma)$ is achieved by a matched pair, } \\
    2,              & \text{if ${\rm cost}(\sigma)$ is achieved by an unmatched point .}
\end{cases}
\]
 Also, for any other line $L'$ in the same equivalence class as $L$, we simply write $\Gamma$ for  the bijection $\Gamma_{L, L'}$  between matchings along $L$ and matchings along $L'$ from \Cref{lem:Gamma}.

As we are assuming that $M,N$ satisfy Property (Q), for $L$ equivalent to $L'$ via $\sim_{\closure C_M\cup \closure C_N\cup\overline{\Omega}}$, $\Gamma$ gives a bijection between the matchings of $L$ to those of $L'$ and, if $u,v$ achieve the cost of $\sigma$, then $u,v$ also achieve the cost of $\Gamma(\sigma)$.

In this situation, the following lemmas tell us how to move a line around in its equivalence class with respect to $\closure C_M\cup \closure C_N \cup 
\closure\Omega$ in order to increase (or at least not decrease) the cost of a matching.

We first analyze what happens when we translate a line.

\begin{lemma}\label{lem:translations-cost}
    Let $L\subseteq \RR^2$ be a line with positive slope and let $\sigma \colon \dgm M^L \to \dgm N^L$ be a matching such that $c(\sigma)=\frac{|p_L(u)-p_L(v)|}{\delta}$, with $u,v\in {\closure C_M\cup \closure C_N\cup\Omega_p}$ and $\delta\in\{1,2\}$. Then, for any other line $L'\sim_{\closure C_M\cup \closure C_N\cup\closure\Omega}L$, obtained by translating $L$, the cost of $\Gamma(\sigma)$ is given by the following formulas:
    \begin{enumerate}
       \item $\cost(\Gamma(\sigma))=\cost(\sigma)=\hat m^L\frac{|p_L(u)-p_L(v)|}{\delta}$, if $A^L_u\subseteq A^{L}_v$ or $A^L_v\subseteq A^{L}_u$, 
       \item $\cost(\Gamma(\sigma)) = \hat m^L\frac{|p_L(u) - p_L(v) + (b_1 - b_1')(1/m_1 + 1/m_2)|}{\delta}$,
 otherwise.
   \end{enumerate}
 \end{lemma}

  \begin{proof}
  Suppose that $L$ is parameterized as $\vec{m}s + b$ and $L'$ as $\vec{m'}s' + b'$.  Let $L'$ be a line that is parallel to and in the same reciprocal position as $L$ with respect to $\closure C_M\cup \closure C_N \cup \closure\Omega$.   
    Since $L'$ is parallel to $L$, we have $\vec{m} = \vec{m'}$, and in particular, $\hat{m}^L = \hat{m}^{L'}$.
   
  As a consequence of Property (Q), we see that
    \[c(\Gamma(\sigma)) = \frac{|p_{L'}(u) - p_{L'}(v)|}{\delta}.\]

                If either $u$ or $v$ lies on $L$, meaning $A^L_u$ or $A^L_v=\{1,2\}$, then $L'\sim_{\closure C_M\cup \closure C_N\cup\closure\Omega}L$ and $L'$ parallel to $L$ implies that $L=L'$, and therefore $c(\Gamma(\sigma))=c(\sigma)$.  Otherwise, we have that both $A^L_u$ and $A^L_v$ are strictly contained in $\{1,2\}$. We now have two cases.

 \begin{enumerate}
    \item First suppose that $L$ (and hence $L'$) intersects the same face of the positive cone of $u$ and of $v$, that is $A^L_u\subseteq A^{L}_v$ or $A^L_v\subseteq A^{L}_u$. 
    
      By \Cref{eq:s'}, this means that there is some $i\in\{1,2\}$ such that:
      \begin{align*}
        p_{L'}(u) &= \frac{m_i}{m'_i} p_L(u) + \frac{b_i - b_i'}{m_i'} = p_L(u) + \frac{b_i - b_i'}{m_i}, \text{ and}\\
        p_{L'}(v) &= \frac{m_i}{m'_i} p_L(v) + \frac{b_i - b_i'}{m_i'} = p_L(v) + \frac{b_i - b_i'}{m_i}.
      \end{align*}
      We can then compute directly:
      \begin{align*}
        c(\Gamma(\sigma)) &= \frac{|p_{L'}(u) - p_{L'}(v)|}{\delta}\\
                          &= \frac{|p_L(u) + (b_i - b_i')/m_i - p_L(v) - (b_i - b_i')/m_i|}{\delta}\\
                          &=\frac{|p_L(u) - p_L(v)|}{\delta}\\
        &= c(\sigma).
      \end{align*}

    \item Now suppose that $L$ (and hence also $L'$) intersects different  faces of the positive cones of the points $u$ and $v$.
      Without loss of generality, assume that $u$ lies strictly to the right of both $L$ and $L'$, and that $v$ lies strictly to the left. 
      
      This means that
      \begin{align*}
        p_{L'}(u) &= p_L(u) + \frac{b_1 - b_1'}{m_1},\\
        p_{L'}(v) &= p_L(v) + \frac{b_2 - b_2'}{m_2}.
      \end{align*}
      Recall that because of our standard normalization  (see \Cref{eq:normalization}), we have $b_1 = -b_2$ and $b_1' = - b_2'$.
      So we have
      \begin{align*}
        c(\Gamma(\sigma)) &= \frac{|p_{L'}(u) - p_{L'}(v)|}{\delta}\\
                          &= \frac{|p_L(u) + (b_1 - b_1')/m_1 - p_L(v) - (b_2 - b_2')/m_2|}{\delta}\\
        &=\frac{|p_L(u) - p_L(v) + (b_1 - b_1')(1/m_1 + 1/m_2)|}{\delta}.
      \end{align*}
    \end{enumerate}
  \end{proof}

  \begin{lemma}\label{lem:translations}
   Let $L\subseteq \RR^2$ be a line with positive slope not intersecting $\closure C_M\cup \closure C_N \cup \Omega_P$. For each matching $\sigma \colon \dgm M^L \to \dgm N^L$, there exists a direction of translation such that if $L'$ is a line obtained by translating $L$ in that direction and $L\sim_{\closure C_M\cup \closure C_N\cup\overline{\Omega}} L'$, then
      $$\cost(\Gamma(\sigma))\geq\cost(\sigma).$$
 \end{lemma}
  
  \begin{proof}
With reference to the notations and cases analyzed in \Cref{lem:translations-cost}, we now have two cases.
    \begin{enumerate}
    \item First suppose that $L$ (and hence $L'$) intersects the same face of the positive cone of $u$ and of $v$.
       We  then know that $\cost(\Gamma(\sigma)) = \cost(\sigma)$.
      This shows that in this case, translating $L$ in either of the two directions keeps the cost unchanged.
      Since we are assuming that $\closure C_M\cup \closure C_N \cup \Omega_p$ contains at least the proper points $u$ and $v$, at least one of the two directions of translation guarantees hitting a proper point in $\closure C_M\cup \closure C_N\cup \Omega_P$.
      
    \item Now suppose that $L$ (and hence also $L'$) intersects different positive faces of the points $u$ and $v$ so that
     $\cost(\Gamma(\sigma)) = \hat m^L\frac{|p_L(u) - p_L(v) + (b_1 - b_1')(1/m_1 + 1/m_2)|}{\delta}$.
      Note that $m_1, m_2 > 0$.
      If $p_L(u) \geq p_L(v)$, then having $b_1 > b_1'$ would imply that $\cost(\Gamma(\sigma)) \geq \cost(\sigma)$.
      If $p_L(v) \geq p_L(u)$, then having $b_1' > b_1$ would imply that $\cost(\Gamma(\sigma)) \geq \cost(\sigma)$.

      Note that $b_1 > b_1'$ means that $L'$ is a left translate of $L$, and $b_1' > b_1$ means that $L$ is a right translate of $L$.
      In either case, there is a direction of translation in which the cost does not decrease.
      We may therefore translate $L$, increasing $c(\sigma)$, until $L$ hits a point in $\closure C_M\cup \closure C_N \cup \Omega_p$.
      Since $u$ and $v$ lie on opposite sides of $L$, we are guaranteed to hit a proper point in $\closure C_M\cup \closure C_N\cup \Omega_P$.
    \end{enumerate}
  \end{proof}

\Cref{lem:translations-cost} is used in the following way. 
 
\begin{remark}
          For each line $L$ with positive slope that does not intersect any point of $\closure C_M\cup \closure C_N\cup \Omega_P$, there is a direction of translation such that

          \begin{itemize}
       \item translating $L$ in that direction does not decrease the weighted cost $\cost(\sigma)$, 
       \item during the translation the line remains in the same equivalence class as $L$ with respect to $\closure C_M\cup \closure C_N\cup \overline{\Omega}$, until
       \end{itemize}
     
\begin{itemize}
       \item eventually, the translated line hits a proper point in $\closure C_M\cup \closure C_N \cup \Omega_P\subseteq \RR^2$.
   \end{itemize}
\end{remark} 

Let us now study what happens when we rotate a line.

\begin{lemma}\label{lem:rotations-cost}
    Let $L$ be a line through a point $r \in \RR^2$ with direction $\vec{m}=(m_1,m_2)$ (with $m_1,m_2 >0$), and let $\sigma \colon \dgm M^L \to \dgm N^L$ be a matching such that $c(\sigma)=\frac{|p_L(u)-p_L(v)|}{\delta}$, with $u,v\in {C_M}\cup C_N$ and $\delta\in\{1,2\}$. Then, for any other line $L'\sim_{\closure C_M\cup \closure C_N\cup\overline{\Omega}}L$, obtained by rotating $L$ about $r$, the cost of $\Gamma(\sigma)$ is given by the following formulas:
    \begin{enumerate}
       \item $\cost(\Gamma(\sigma)) = \frac{\hat{m}^{L'}}{\hat{m}^L}\cdot \cost(\sigma),$ if both $A^L_u$ and $A^L_v$  contain ${\argmax}_{i\in \{1,2\}}\{m_i\}$. 
       
       \item $\cost(\Gamma(\sigma)) = \cost(\sigma)$ if both $A^L_u$ and $A^L_v$ contain ${\argmin}_{i\in \{1,2\}}\{m_i\}$. 
       
       \item $\cost(\Gamma(\sigma)) = \frac{1}{\delta} \cdot \left(\hat{m}^{L'}(v_2 - r_2) + r_1- u_1 \right),$ if $A^L_u={\argmin}_{i\in \{1,2\}}\{m_i\}$, $A^L_v = {\argmax}_{i\in \{1,2\}}\{m_i\}$ and $p_{L'}(u) - p_{L'}(v) \leq 0$.
        
    \item $\cost(\Gamma(\sigma)) = \frac{1}{\delta}\cdot \left(\hat{m}^{L'}(r_2 - v_2) +  u_1 - r_1 \right),$ if $A^L_u={\argmin}_{i\in \{1,2\}}\{m_i\}$, $A^L_v = {\argmax}_{i\in \{1,2\}}\{m_i\}$ and $p_{L'}(u) - p_{L'}(v) \geq 0$.
      
   \end{enumerate}
 \end{lemma}

\begin{proof}
Assume, without loss of generality, that $L$ has slope greater than $1$. That is, $\hat m^{L}=m_1<m_2=1$, so that ${\argmax}_{i\in \{1,2\}}\{m_i\} = 2$ and ${\argmin}_{i\in \{1,2\}}\{m_i\} = 1$. 
\begin{enumerate}
        \item Suppose that $A^L_u$ and $A^L_v$ both contain $2$.
          Then
          \[|p_L(u) - p_L(v)| = |u_2 - v_2| = |p_{L'}(u) - p_{L'}(v)|.\]
          Therefore,
          \[\cost(\Gamma(\sigma)) = \hat{m}^{L'} \frac{|p_L(u) - p_L(v)|}{\delta} = \frac{\hat{m}^{L'}}{\hat{m}^L}\cdot \cost(\sigma).\]

        \item Next suppose that both $A^L_u$ and $A^L_v$ contain $1$.
          Then
          \[|p_{L'}(v) - p_{L'}(u)| = \frac{|v_1 - u_1|}{m_1^{L'}}.\]
          So we have that
          \[\cost(\Gamma(\sigma)) = \hat{m}^{L'}\cdot \frac{|p_{L'}(v) - p_{L'}(u)|}{\delta} = \frac{1}{\delta}\cdot m_1^{L'} \cdot \frac{|v_1 - u_1|}{m_1^{L'}} = \frac{|v_1 - u_1|}{\delta} = \cost(\sigma).\]
          
      \item Assume that $A^L_u = \{1\}$ and $A^L_v = \{2\}$.
        That is, assume that $u$ lies strictly to the right of the line $L$ and that $v$ lies strictly to the left of the line $L$.
        Let $L'$ be another line obtained by rotating $L$ around $r$ without hitting any point in $\closure C_M\cup \closure C_N \cup \overline{\Omega}$.
        We have $A^{L'}_u = A^L_u$ and $A^{L'}_v = A^L_v$.
        Moreover,
        \[\frac{\push_{L'}(u)_2 - r_2}{\push_{L'}(u)_1 - r_1} = \frac{1}{\hat{m}^{L'}} = \frac{1}{m_1^{L'}}.\]
        Thus,
        \[\push_{L'}(u)_2 = \frac{\push_{L'}(u)_1 - r_1}{m_1^{L'}} + r_2 = \frac{u_1 - r_1}{m_1^{L'}} + r_2.\]
        By \Cref{eq:param-diff}, we therefore see that
        \begin{align*}
          p_{L'}(v) - p_{L'}(u) &= \push_{L'}(v)_2 - \push_{L'}(u)_2\\
                                &= v_2 - \left( \frac{u_1 - r_1}{m_1^{L'}} + r_2 \right)\\
          &= v_2 - r_2 + \frac{r_1 - u_1}{m_1^{L'}}.
        \end{align*}
        Hence, if $p_{L'}(v) - p_{L'}(u) \geq 0$, the cost of $\Gamma(\sigma)$ for $L'$ is
        \begin{align*}
          \cost(\Gamma(\sigma)) &= \hat{m}^{L'}\cdot \frac{p_{L'}(v) - p_{L'}(u)}{\delta}\\
                                &=\frac{1}{\delta}\cdot m_1^{L'}\cdot \left(v_2 - r_2 + \frac{r_1- u_1}{m_1^{L'}}\right)\\
          &= \frac{1}{\delta} \cdot \left(m_1^{L'}(v_2 - r_2) + r_1- u_1 \right).
        \end{align*}
        
        \item Similarly, if $p_{L'}(u) - p_{L'}(v) \geq 0$, then the cost of $\Gamma(\sigma))$ for $L'$ is
        \begin{align*}
          \cost(\Gamma(\sigma)) = \frac{1}{\delta}\cdot \left(m_1^{L'}(r_2 - v_2) +  u_1 - r_1 \right) \ .
          \end{align*}
\end{enumerate}

\end{proof}


\begin{lemma}\label{lem:rotations}
      Let $L\subseteq \RR^2$ be a line with positive slope and let $r$ be any proper point on $L$. 
      If $L$ does not pass through any point of $\closure C_M\cup \closure C_N\cup \overline{\Omega}$, apart from possibly the point $r$ itself, then, for each matching $\sigma \colon \dgm M^L \to \dgm N^L$, there exists a direction of rotation such that if $L'$ is a line obtained by rotating $L$ around $r$ in this direction and $L\sim_{\closure C_M\cup \closure C_N\cup\overline{\Omega}} L'$, then 
      $$\cost(\Gamma(\sigma))\geq\cost(\sigma).$$
    \end{lemma}
    
    \begin{proof}
    Assume that $L$ does not pass through any point of $\closure C_M\cup \closure C_N\cup \overline{\Omega}$ other than $r$. 
      Recall that the point $[0:1:1]$ lies in $\overline{\Omega}$, so by our assumptions, $L$ is not a line with diagonal direction.
       Assume, without loss of generality, that the slope of $L$ is larger than $1$, that is $\hat{m}^L =m_1<m_2=1$.
      Consider the matching $\sigma$, and suppose that
      \[c(\sigma) = \frac{|p_L(u) - p_L(v)|}{\delta}.\]
%
      By assumption, one of the following mutually exclusive cases applies:
      \begin{itemize}
      \item  at least one of the points $u$ and $v$ is not on $L$, and $r$ equals the other point, 
      \item neither of the points $u$ and $v$ lie on $L$.
      \end{itemize}

      We therefore have the following cases.
      \begin{enumerate}
      \item The points $u$ and $v$ lie on the same side of $L$, including the possibility that one of them lies on $L$.
        In other words, either $A^L_u$ and $A^L_v$ are equal, or one is a proper subset of the other one.
      \item The points $u$ and $v$ lie strictly on different sides of $L$, implying that neither of them lie on $L$.
        In other words, one of $A^L_u$ and $A^L_v$ equals $\{1\}$, and the other one equals $\{2\}$.
      \end{enumerate}

      We prove each case separately.
      \begin{enumerate}
      \item First suppose that $u$ and $v$ lie on the same side of $L$.
        This means that either both $A^L_u$ and $A^L_v$  contain $1$ or both contain $2$.
        Let $L'$ be a line in the same equivalence class as $L$ with respect to $\closure C_M\cup \closure C_N \cup \overline{\Omega}$.
        
        We have two subcases.
        \begin{enumerate}
        \item Suppose that $A^L_u$ and $A^L_v$ both contain $2$.
          Then, by \Cref{lem:rotations-cost}(1), we have that 
          \[\cost(\Gamma(\sigma)) = \frac{\hat{m}^{L'}}{\hat{m}^L}\cdot \cost(\sigma).\]
          Therefore rotating $L$ towards the diagonal will increase the cost, at least until we hit a point of $\closure C_M\cup \closure C_N \cup \overline{\Omega}$, possibly the diagonal direction.

        \item Next suppose that both $A^L_u$ and $A^L_v$ contain $1$.
          Then, by \Cref{lem:rotations-cost}(2), we have 
          \[\cost(\Gamma(\sigma)) = \cost(\sigma).\]
          So the cost of the matching does not change as we rotate $L$ about $r$.
          We can once again rotate towards the diagonal without decreasing the cost, at least until we hit a point in $\closure C_M\cup \closure C_N \cup \overline{\Omega}$, including the diagonal direction.
        \end{enumerate}

      \item Next suppose that $u$ and $v$ lie on different sides of $L$.
        Assume, without loss of generality, that $A^L_u = \{1\}$ and $A^L_v = \{2\}$.
        That is, $u$ lies strictly to the right of the line $L$ and $v$ lies strictly to the left of the line $L$.
        
        Let $L'$ be another line obtained by rotating $L$ around $r$ without crossing any point in $\closure C_M\cup \closure C_N \cup \overline{\Omega}$.   The sign of  $p_{L'}(u) - p_{L'}(v)$ is determined by $A^L_{\lub(u,v)}$.

        \begin{itemize}
        \item If $\{1\} \subseteq A^L_{\lub(u,v)}$, then $p_{L'}(u) - p_{L'}(v) \geq 0$. By \Cref{lem:rotations-cost}(4), we have 
        \begin{equation}\label{eq:uv2}
        \cost(\Gamma(\sigma)) = \frac{1}{\delta}\cdot \left(m_1^{L'}(r_2 - v_2) +  u_1 - r_1 \right).
        \end{equation}
        
        \item If $\{2\} \subseteq A^L_{\lub(u,v)}$, then $p_{L'}(u) - p_{L'}(v) \leq 0$. By \Cref{lem:rotations-cost}(3), we have 
         \begin{equation}\label{eq:vu2}
          \cost(\Gamma(\sigma)) = \frac{1}{\delta} \cdot \left(m_1^{L'}(v_2 - r_2) + r_1- u_1 \right).
          \end{equation}
 	\end{itemize}
        
        We now have the following further subcases to consider.

        \begin{enumerate}
        \item Suppose that $\{1\} \subseteq A^L_{\lub(u,v)}$ and $v_2 \geq r_2$.
          By~\Cref{eq:uv2}, taking $m_1^{L'} < m_1$ rotates $L$ away from the diagonal direction and does not decrease the cost of the matching $\sigma$. 
          In this case, we may hit no point of $\closure C_M\cup \closure C_N \cup \overline{\Omega}$ before we get to the vertical line through $r$, that is, we may eventually hit the point $[0:0:1]\in \ell_\infty$.

        \item Now, suppose that $\{1\} \subseteq A^L_{\lub(u,v)}$ and $v_2 < r_2$.
          This time, we take $m_1^{L'} > m_1$, which rotates the line $L$ towards the diagonal direction without decreasing its cost. 
        In this case, we may hit no point of $\closure C_M\cup \closure C_N \cup \overline{\Omega}$ before we get to the diagonal line through $r$, that is, we may eventually hit the point $[0:1:1]\in \overline{\Omega}$.
        
        \item Next, suppose that $\{2\} \subseteq A^L_{\lub(u,v)}$ and $v_2 > r_2$.
          So by~\Cref{eq:vu2}, we see that taking $m_1^{L'} > m_1$, that is, rotating the line $L$ towards the diagonal slope, does not decrease the cost.

        \item Finally, suppose that $\{2\} \subseteq A^L_{\lub(u,v)}$ and $v_2 \leq r_2$.
          This time, we can take $m_1^{L'} < m_1$, which rotates the line away from the diagonal slope.
          This does not decrease the cost, and we will hit the point $u$  before getting to a vertical line. 

\end{enumerate}

              \end{enumerate}
           When  the slope of $L$ is smaller than $1$, that is $\hat{m}^L =m_2<m_1=1$, the argument is very similar except that 
          for the case 2(c) where we may hit the point $[0:1:0]\in \ell_{\infty}$ instead of the point $[0:0:1]$. 
      The case when the slope of $L$ is exactly 1 is excluded by the assumption that $L$ intersects $\closure C_M\cup \closure C_N\cup \overline{\Omega}
      $ only at the proper point $r$.
    \end{proof}
    
\Cref{lem:rotations} is used in the following way.

 \begin{remark}
          For each line $L$ as in \Cref{lem:rotations}, there is a direction of rotation such that  
      \begin{itemize}
      \item rotating $L$ in such direction around $r$ does not decrease $\cost(\sigma)$, 
      \item during the rotation the line remains in the same equivalence class as $L$ with respect to $\closure C_M\cup \closure C_N\cup \overline{\Omega}$,
      \end{itemize}
      until 
      \begin{itemize}
          \item eventually, the rotated line hits a point in $\closure C_M\cup \closure C_N \cup \overline{\Omega}\cup \{[0:1:0], [0:0:1]\}$.
       \end{itemize}
       
     \end{remark}

\begin{remark}[Diagonal Lines]\label{rem:diaglines}
    In \Cref{lem:rotations}, diagonal lines that do not appear naturally as lines through two critical or switch points are artificially added as a conventional stopping criterion during a rotation.  We believe it may be possible to prove our results without adding $[0:1:1]$ as a point in $\overline{\Omega}$.  However, in the next section, we will prove that this choice also allows us to discard vertical and horizontal lines. 
\end{remark}

\subsection{Vertical and Horizontal Lines}\label{sec:VandHlines}

As can be seen in the proof of \Cref{lem:rotations}, it is not immediately apparent that the bottleneck distance for horizontal or vertical lines is irrelevant for the overall matching distance computation - in other words, this lemma does not discount the possibility that the matching distance is achieved as a limit to such a line.  To understand better the role that horizontal and vertical lines play in the computation of the matching distance, we must more formally derive the cost of a vertical line (\Cref{def:verticalcost}); the correctness of this definition relies on \Cref{thm:vertical}.
After defining the cost of a vertical line, we are able to show that these lines do not determine the matching distance, via \Cref{cor:no-vertical}.  The case dealing with horizontal lines is analogous.

In this section, we maintain the notations and assumptions of the previous section. In particular, the persistence modules $M$ and $N$ are assumed to satisfy Property (P)  for some finite sets $C_M$ and $C_N$, and Property (Q) for some set $\Omega$. Moreover, we maintain the notation $\overline{\Omega}:=\Omega\cup \{[0:1:1]\}$.

In order to prove \Cref{thm:vertical}, we first prove the following two lemmas concerning the cost of lines within the same equivalence class as a vertical line. 

\begin{lemma}\label{lem:steepenoughline}
Let $V\subseteq \RR^2$ be a vertical line  and let $c, c'  $ be distinct points on $V$ such that the open line segment on $V$ between them does not intersect $\closure C_M\cup \closure C_N \cup \overline{\Omega}$. Then there exists a $\tilde{m}_1>0$ sufficiently small such that, for any two parallel  lines $L$ and $L'$ with direction $\vec m=( m_1, 1)$ and $0<m_1<\tilde m_1$ that intersect $V$ on the open line segment $c,c'$, it holds that $L\sim_{\closure C_M\cup \closure C_N\cup\overline{\Omega}}L'$.
\end{lemma}

\begin{proof}
Let us set
\begin{equation}
\begin{split}
\varepsilon_{\infty} := &\min \{m_1\in\RR\mid [0:m_1:m_2]\in \overline{\Omega}, m_2 = 1 \} \ , \\
c_{min} := & {\glb}\left(\closure C_M\cup \closure C_N \cup \Omega_P \right) \ , \\
c_{max} := & {\lub}\left(\closure C_M\cup \closure C_N \cup \Omega_P \right) \ , \\
\varepsilon := &\min \{|x_1-y_1| \mid  x_1 \neq y_1 \, , \, x, y\in \closure C_M\cup \closure C_N \cup \Omega_P\} \ . \\
\end{split}
\end{equation}

Assume, without loss of generality, that $c_2 < c'_2$. Taking $\varepsilon/0$ to be equal to $+\infty$, choose
\[
0<\tilde{m}_1 < \min \left\{\frac{\varepsilon}{|c_2' - c_2^{min}|}, \frac{\varepsilon}{|c_2^{max} - c_2|}, \varepsilon_{\infty} \right\} \ .
\]
As there are no points in $\closure C_M\cup \closure C_N \cup \overline{\Omega}$ on $V$ between $c$ and $c'$, we get that $L$ and $L'$ have the same reciprocal position with respect to all points in $V \cap (\closure C_M\cup \closure C_N \cup \overline{\Omega})$. Moreover, by the choice of their slope, $L$ and $L'$ are in the same reciprocal position with respect to all points in 
$(\closure C_M\cup \closure C_N \cup \overline{\Omega}) \setminus V$ as well.  Thus, $L\sim_{\closure C_M\cup \closure C_N \cup \overline{\Omega}}L'$.
\end{proof}

\begin{lemma}\label{lem:same-limit-local}
Let $V\subseteq \RR^2$ be a vertical line, and let $c, c'$ be distinct points on $V$ such that the line segment between them does not contain any point of  $\closure C_M\cup \closure C_N \cup \overline{\Omega}$.
Let $L$ be a line through a point $r\in V$ between $c$ and $c'$ with direction $\vec{m} = (m_1,1)$, $0<m_1< 1$.
 Then the limit of the weighted bottleneck distance along $L$ as it rotates around $r$ toward the vertical direction, written as
  \[\lim_{m_1 \to 0^+} \hat m^L d_B(\dgm M^L, \dgm N^L),\]
  is the same for all points $r$ in the open segment of $V$ between $c$ and $c'$.  Moreover, this convergence is uniform on the open line segment between $c$ and $c'$.
\end{lemma}

\begin{proof}
We consider the behavior of the weighted bottleneck distance along  lines $L(r,m_1)$ with direction $\vec{m} = (m_1,1)$, and $0<m_1<1$, through a point $r$ on the open segment of $V$ between $c$ and $c'$, as $m_1$ tends to $0$.

First of all we observe that,  for any two points $r$ and $r'$ on such open segment, and for sufficiently small $m_1>0$,  \Cref{lem:steepenoughline} ensures that $L:=L(r,m_1)\sim_{\closure C_M\cup \closure C_N \cup \overline{\Omega}} L(r',m_1)=:L'$. Hence, by Property (Q'),  if the optimal matching achieving the bottleneck distance along $L$ is $\sigma$, then the optimal matching achieving the bottleneck distance along $L'$ is $\Gamma_{L,L'}(\sigma)$, and that the matched pair giving the cost of $\Gamma_{L,L'}(\sigma)$ is the image of the matched pair giving the cost of $\sigma$ under $\Gamma_{L,L'}$: If $p_{L}(u),p_{L}(v)$ give the bottleneck distance along $L$, then $p_{L'}(u),p_{L'}(v)$ give the bottleneck distance along $L'$. In other words, for any line $L(\cdot, m_1)$ that is steep enough and intersects $V$ between $c$ and $c'$, the bottleneck distance is determined by the same pair $u,v$ of critical values.  

Now, considering the cases listed in \Cref{lem:rotations-cost} separately, we show that the weighted bottleneck distance along $L(r,m_1)$ converges uniformly, independent of the choice of $r$, as $m_1$ tends to 0. That is, for any $\epsilon > 0$, there exists a $\tilde{m}_1>0$ such that for any $0<m_1 < \tilde{m}_1$ and any point $r' \in V$  in between $c$ and $c'$, the line $L'$ through $r''$  with direction $(m_1,1)$ gives
\[
|\hat m^{L'} d_B(\dgm M^{L'}, \dgm N^{L'}) - \lim_{m_1 \to 0^+} \hat m^L d_B(\dgm M^L, \dgm N^L)| < \epsilon \ . 
\]

Let $\epsilon > 0$.

\begin{enumerate}
    \item If both $A_v^{L(r,m_1)}$ and $A_u^{L(r,m_1)}$ contain $2={\argmax}_{i\in \{1,2\}}\{m_i\}$, then the weighted bottleneck distance is $m_1\frac{|u_2-v_2|}{\delta}$, which is independent of $r$ and tends to $0$ when $m_1$ tends to $0$. In particular, for any $r \in V$ between $c$ and $c'$ and any $0 < m_1 < \frac{\epsilon}{|u_2-v_2|}$, we have that $m_1\frac{|u_2-v_2|}{\delta}< \epsilon$, proving the uniform convergence of $\hat m^{L(r,m_1)} d_B\left(\dgm M^{L(r,m_1)}, \dgm N^{L(r,m_1)}\right)$ to $0$ as $m_1$ tends to $0$ in this case. 
    
    \item If  $A_v^{L(r,m_1)}$ and $A_u^{L(r,m_1)}$ contain $1={\argmin}_{i\in \{1,2\}}\{m_i\}$, then the weighted bottleneck distance is constantly equal to $\frac{|u_1-v_1|}{\delta}$, independently of both $r$ and $m_1$. Hence, it converges uniformly (for steep enough lines intersecting $V$ between $c$ and $c'$) to $\frac{|u_1-v_1|}{\delta}$ as $m_1$ tends to $0$  in this case. 
    
    \item If $A_u^{L(r,m_1)}=\{1\}$ and $A_v^{L(r,m_1)}=\{2\}$, and $p_{L(r,m_1)}(u)-p_{L(r,m_1)}(v) \leq 0$, then the weighted bottleneck distance is $\frac{1}{\delta} \cdot \left(m_1(v_2 - r_2) + r_1- u_1 \right)$. As $m_1$ tends to $0$, this tends to $\frac{1}{\delta} \cdot (r_1-u_1)$ regardless of $r$ (that is, regardless of $r_2$, as $r_1$ is the same for any point $r$ between $c$ and $c'$), thanks to the fact that $r_2$ is bounded between $c_2$ and  $c'_2$.
        
    In particular, we have that 
    \begin{equation*}
    \begin{split}
   \left | m_1 d_B\left(\dgm M^{L(r,m_1)}, \dgm N^{L(r,m_1)}\right) - \frac{1}{\delta} \cdot (r_1-u_1)\right| = m_1 \cdot (v_2-r_2)\\ < m_1 \cdot \max\{|v_2-c_2|, |v_2-c_2'|\} \ .
    \end{split}
    \end{equation*}
    Therefore,  for any $r \in V$ between $c$ and $c'$ and any $0 < m_1 < \frac{\epsilon}{\max\{|v_2-c_2|, |v_2-c_2'|\}}$, we have 
    $\left|m_1\cdot\frac{|u_2-v_2|}{\delta}- \frac{1}{\delta} \cdot (r_1-u_1)\right| < \epsilon$, 
    proving uniform convergence of the weighted bottleneck distance along $L(r,m_1)$ to $\frac{1}{\delta} \cdot (r_1-u_1)$ as $m_1$ tends to $0$ in this case. 
    
    \item If $A_u^{L(r,m_1)}=\{1\}$ and $A_v^{L(r,m_1)}=\{2\}$, and $p_{L(r,m_1)}(u)-p_{L(r,m_1)}(v) \geq 0$, then, by an analogous argument to the one made in the previous case, we can see that the weighted bottleneck distance along $L(r,m_1)$ converges uniformly to $\frac{1}{\delta} \cdot (u_1-r_1)$ as $m_1$ tends to $0$ in this case. 
   
   \end{enumerate}

\end{proof}

\begin{theorem}\label{thm:vertical}
Given a vertical line $V\subseteq \RR^2$,  the limit \[\lim_{m_1 \to 0^+} \hat m^L d_B(\dgm M^L,\dgm N^L) \]
of the weighted bottleneck distance of the line $L$ through the point  $r\in V$ with direction $\vec m=(m_1,1)$, $0<m_1<1$,  as $L$ rotates toward the vertical direction, does not depend on the point $r$.
\end{theorem}

\begin{proof}
By \Cref{lem:same-limit-local}, we already know that, for every two points $r,r'\in V$ such that the closed segment between them does not contain any point of $\closure C_M\cup \closure C_N\cup \Omega_P$, the limit of the weighted bottleneck distance of lines through the point  $r\in V$  as they rotate toward the vertical direction is equal to that of lines through $r'$. 

Let us now take $c$ to belong to $V\cap (\closure C_M\cup \closure C_N\cup \overline{\Omega})$ and  $r\in V$ to be arbitrarily close to $c$. Let $L(r)$ and $L(c)$ be parallel lines with positive slope through $r$ and $c$, respectively. 

Let us now assume that  $\hat m^{L(r)} d_B(\dgm M^{L(r)},\dgm N^{L(r)})$ tends to the value $\ell$, which has to be the same for all $r$ sufficiently close to $c$ by \Cref{lem:same-limit-local}, and  $\hat m^{L(c)} d_B(\dgm M^{L(c)},\dgm N^{L(c)})$ tends to the value $\ell(c)$, as $L(c)$ and $L(r)$ tend to $V$ by rotating counterclockwise about $r$ and $c$, respectively. Let us also assume, by contradiction, that  $|\ell(c)-\ell|>0$ and take $0<\varepsilon<|\ell(c)-\ell|/3$. 

By internal stability  \cite[Thm. 2]{Landi2018}, we know that the weighted bottleneck distance for $L(r)$ tends to that of $L(c)$ as $r$ tends to $c$, so 
\begin{equation}\label{eq:thm-vert-1}
\left|\hat m^{L(r)} d_B\left(\dgm M^{L(r)},\dgm N^{L(r)}\right)-\hat m^{L(c)} d_B\left(\dgm M^{L(c)},\dgm N^{L(c)}\right)\right|\le \varepsilon
\end{equation}
for every $r$ sufficiently close to $c$.

By \Cref{lem:same-limit-local}, we also know that the weighted bottleneck distance for $L(r)$ tends  to $\ell$ uniformly as we rotate $L(r)$ towards the vertical line and move $r$ on the line. Hence, for all $r$ sufficiently close to $c$, there is some real number  $\tilde m_1>0$ such that for all $0<m_1\le \tilde m_1$, the lines $L(r)$ through $r$ with direction $\vec m=(m_1,1)$ satisfy 
\begin{equation}\label{eq:thm-vert-2}
\left|\hat m^{L(r)} d_B\left(\dgm M^{L(r)},\dgm N^{L(r)}\right)-\ell\right|\le \varepsilon
\end{equation}
We may now take $m_1>0$ to be sufficiently small so that for the lines $L(c)$ through $c$ with direction $\vec m=(m_1,1)$ we also have
\begin{equation}\label{eq:thm-vert-3}
\left|\hat m^{L(c)} d_B\left(\dgm M^{L(c)},\dgm N^{L(c)}\right)-\ell(c)\right|\le \varepsilon
\end{equation}

By applying the triangle inequality, and using \Cref{eq:thm-vert-1,eq:thm-vert-2,eq:thm-vert-3}, we get $|\ell(c)-\ell|\le 3\varepsilon< |\ell(c)-\ell|$, giving a contradiction.
\end{proof}

The above preposition guarantees that the following notion for the cost of a vertical line is well defined.  The case of the analogous statement for horizontal lines can be handled in a similar way.

\begin{definition}\label{def:verticalcost}
The \emph{cost  of a vertical line $V$} is the limit of the weighted bottleneck distance of  a line $L$ through a fixed point $r\in V$  as it rotates toward the vertical direction:
\[\cost(V):=\lim_{\hat m^L \to 0^+} \hat m^L d_B(\dgm M^L, \dgm N^L)\]
where $L$ is any line  through $r\in V$ with direction $\vec m=(m_1,m_2)$ and $\hat m^L=m_1$. 
\end{definition}

We can now conclude that the matching distance can always be achieved by a non-vertical and non-horizontal line, as is shown in the following corollary.

\begin{corollary}\label{cor:no-vertical}
For every vertical line $V$, there exists a line $L$ of positive slope through a point of $\closure C_M\cup \closure C_N \cup \Omega_P$ and a different point of $\closure C_M\cup \closure C_N \cup \overline{\Omega}$ such that
\[
\cost(V)\le \hat m^L d_B(\dgm M^L, \dgm N^L) \ .
\]
An analogous statement is true for horizontal lines. 
\end{corollary}

\begin{proof}
By \Cref{thm:vertical} the cost of a vertical line $V$ does not depend on the point $r\in V$ used to compute it. Thus, we are free to choose $r\in V$ such that it is higher than any point in $\closure C_M\cup \closure C_N\cup \Omega_P$. We can now rotate $V$ about $r$ making sure we stop at the first point of $\closure C_M\cup \closure C_N\cup \overline{\Omega}$ we hit.  The proof of \Cref{lem:rotations} ensures that, if we rotate towards the diagonal, the obtained line $L$ has a weighted bottleneck distance not smaller than $\cost(V)$ except possibly when we are in case 2.(a) or 2.(d).  However, since we have chosen $r$ higher than any point in $\closure C_M\cup \closure C_N\cup \Omega_P$, we cannot be in case 2.(a).


Finally, suppose that when we slightly rotate $V$ around $r$ toward the diagonal direction we get lines $L'$ that fall in case 2.(d). As usual, we call $u,v$ the two points realizing the weighted bottleneck distance along such lines. Let us recall that, in case 2.(d), $u$ lies strictly to the right of $L'$, $v$ lies strictly to the left of $L'$,  and $\lub(u,v)$ lies to the left of $L'$. However, if $\lub(u,v)$ lies to the left of $L'$ and $r_2>v_2,u_2$, it must be the case that $r_1\geq u_1$, since $r$ is a point on the line $L'$.  This would imply that $u$ does not lie strictly to the right of $V$, the vertical line through $r$, which is a contradiction.  So, we cannot be in case 2.(d).

If the obtained line $L$ intersects $\overline{\Omega}$ only at a point at infinity, then \Cref{lem:translations} guarantees that we can translate $L$ until we hit a point of $\closure C_M\cup \closure C_N\cup \Omega_P$ without decreasing the weighted bottleneck distance, and we are done.

If the obtained line $L$ is through a point $c$ in $\closure C_M\cup \closure C_N\cup \Omega_P$, then  $c$ must be strictly to the left of $V$ because $r$ was chosen to be strictly higher than $\closure C_M\cup \closure C_N\cup \Omega_P$. Now we can use $c$ as a new center of rotation applying again \Cref{lem:rotations}. If we get again a vertical line, then we can iterate the argument. Eventually, there will be no points strictly to the left of the newly obtained  vertical line, and thus the process will terminate with a line $L$ through a point of $\closure C_M\cup \closure C_N \cup \Omega_P$ and a different point of $\closure C_M\cup \closure C_N \cup \overline{\Omega}$.

\end{proof}

\subsection{Proof of the main theorem}\label{sec:proof-main}

We are now ready to prove \Cref{thm:main}, thanks to \Cref{prop:Q}, \Cref{lem:translations}, \ref{lem:rotations}, and \Cref{cor:no-vertical}, which we restate below for convenience.

\main* 

\begin{proof}
If $M$ and $N$ are both trivial, then $M^L$ and $N^L$ are also trivial for all lines $L$, yielding that $d_{match}(M,N)=0$. Otherwise, if $M$ and $N$ are not both trivial, consider the critical sets $C_M$ and $C_N$ guaranteed by Property (P), and which cannot be both empty. Consider a line $L$ on which $M^L$ and $N^L$ are also not both trivial.  Then, the matching distance between $M$ and $N$ is achieved on one such line $L$. 

If $\dgm(M^L),\dgm(N^L)$ have a different number of points at infinity, then $d_B\left(\dgm(M^L),\dgm(N^L)\right)=\infty$. For any line $L'\sim_{\closure C_M\cup \closure C_N\cup \overline{\Omega}}L$, the points at infinity of $\dgm(M^L)$ (resp. $\dgm(N^L)$) are in bijection with those of $\dgm(M^{L'})$ (resp. $\dgm(N^{L'})$) via $\gamma_M$ (resp. $\gamma_N$) as defined in \Cref{lem:Gamma}. So, $\dgm(M^{L'})$ and $\dgm(N^{L'})$ will also have a different number of points at infinity, implying that $d_B\left(\dgm(M^{L'}),\dgm(N^{L'}\right)=\infty$ as well. By internal stability  \cite[Thm. 2]{Landi2018},  the bottleneck distance is equal to infinity on any line on the closure\footnote{Here, two lines $L$ and $L'$ endowed with the standard parametrization $u=\vec m \cdot t+b$ and $u=\vec m' \cdot t+b'$, respectively, are $\varepsilon$-distant if $\max\{\|\vec m-\vec m'\|_\infty, \|b-b'\|_\infty\}=\varepsilon$.} of the equivalence class of $L$. In this closure, we can always find a line passing through two points, one in $\closure C_M\cup \closure C_N \cup \Omega_P$ and the other one in $\closure C_M\cup \closure C_N \cup \closure\Omega$, implying the claim. 

If  $\dgm(M^L),\dgm(N^L)$ have the same number of points at infinity, then $d_B\left(\dgm(M^L),\dgm(N^L)\right)$ is realized by a matching.
%
%
Properties (P) and (Q') guarantee that  the optimal matching $\sigma^L$ and the pair of points $u,v\in \closure C_M\cup \closure C_N$ whose push onto $L$  gives maximal distance in that matching do not change if we vary the line $L$ remaining inside its equivalence class with respect to $\closure C_M\cup \closure C_N\cup \overline{\Omega}$, i.e. without hitting any new points in $\closure C_M\cup \closure C_N\cup\overline{\Omega}$.

Depending on the following exhaustive list of cases, we see that it is always possible to get a line $L'$ that gives no smaller weighted bottleneck distance and that passes through one point in $\closure C_M\cup \closure C_N \cup \Omega_P$ and another one in $\closure C_M\cup \closure C_N \cup \closure\Omega$. The list of cases to analyze is:

\begin{enumerate}
    \item[(i)] $L$ is already through one point in $\closure C_M\cup \closure C_N \cup \Omega_P$ and another one in $\closure C_M\cup \closure C_N \cup \closure\Omega$,
    \item[(ii)] $L$ is only through one point of $\closure C_M\cup \closure C_N\cup \Omega_P$,
    \item[(iii)] $L$ is only through one point at infinity of $\overline{\Omega}$,
    \item[(iv)] $L$ is  through no point of $\closure C_M\cup \closure C_N \cup \overline{\Omega}$.
\end{enumerate}

In case (i) there is nothing to prove. In case (ii), the claim is proved by applying  \Cref{lem:rotations}, i.e. by rotating the line in the appropriate direction about its point in $\closure C_M\cup \closure C_N\cup \Omega_P$ until we hit another point in $\closure C_M\cup \closure C_N\cup \overline{\Omega}\cup \{[0:1:0],[0:0:1]\}$. If the point we hit is in $\closure C_M\cup \closure C_N\cup \overline{\Omega}$, then we are done, otherwise the line is either vertical or horizontal, and we can apply \Cref{cor:no-vertical} and we are also done. In case (iii), the claim is proven by applying \Cref{lem:translations}, i.e.~by translating the line until we get a line hitting a point of $\closure C_M\cup \closure C_N\cup \Omega_P$ and still through the same point of $\overline{\Omega}_\infty$ as $L$. 
In case (iv), by \Cref{lem:translations}, there exists a direction of translation such that we can translate the line in that direction without decreasing the cost and eventually hit a point in $\closure C_M\cup \closure C_N \cup \Omega_P$. We are now in case (iii) from which we can reach the conclusion as explained above.
\end{proof}

\subsection{Connections to related works}\label{sec:rel-work3}

We highlight the relationship of the contents of \Cref{sec:claims} -\ref{sec:proof-main} to both the  topological approach papers \cite{Cerri2019geometrical, ethier2023geometry}, and the dual approach papers \cite{Kerber-Lesnick-Oudot2018, Bjerkevik2021}.

Within the topological approach, \Cref{lem:rotations-cost} corresponds to Theorem 4 from \cite{Cerri2019geometrical}; both results state when it is possible to continuously move a line through parameter space without decreasing the cost of any matching.  For us, this result is integral in proving that the matching distance is achieved as a maximum, rather than a supremum. 
Differently, in \cite{Cerri2019geometrical}, it is used to prove this realization as a maximum for the {\em coherent matching distance}, a modification of the matching distance restricted to the closure of an open connected subset of lines not passing through a switch point (rephrased in our language).

The use of vertical and horizontal lines in biparametric persistence, as
well as the extension of the bottleneck distance to these lines, is discussed in \cite{ethier2023geometry} Section 4, albeit
from a different perspective. Notably, \Cref{lem:same-limit-local}, \Cref{thm:vertical},
and \Cref{def:verticalcost} are related to the approach described in \cite{ethier2023geometry}. 

Passing to the comparison with \cite{ethier2023geometry}, which appeared on \texttt{arXiv} at the same time as this manuscript (cf. \cite{bapat2022computing}), we observe that the set of switch points $\Omega$ in \Cref{thm:main} is closely related to the set of special values introduced in \cite{ethier2023geometry}  (Definition 5.1). Lines passing through a switch point $\omega$ correspond to the lines $r(a, b)$ associated with a ``special value" $(a, b)$ in the topological case.

Moreover, the statement of \Cref{thm:main} corresponds to Theorem 5.4 in \cite{ethier2023geometry}, although the analogous theorem in the topological case does not exclude vertical and horizontal lines. The proofs of these theorems have similarities as well; a crucial step in proving Theorem 5.4 involves showing that the line $L$ can be continuously moved without reducing the optimal matching cost until it reaches a line of slope $0, 1, \infty,$ or a line associated with a special value. 

 We next highlight similarities and differences of the results of \Cref{sec:proof-main} and \cite{Kerber-Lesnick-Oudot2018}.

In \cite{Kerber-Lesnick-Oudot2018}, the line arrangement $\mathcal{A}_0$ consists of lines in the dual space which correspond to parameterizations of lines in the primal space through critical values.  This line arrangement is further refined to $\mathcal{A}$ (Definition 7), with the goal of obtaining a finite polyhedral decomposition of the dual space for which one may determine a ``consistent matching" for computing the weighted bottleneck distance for all lines in a given polyhedron (Theorem 9).  

Analogously, and in the primal space, the preliminary claims in \Cref{sec:background} are all used to show, in the proof of \Cref{thm:main}, that one may similarly find a ``consistent matching" for computing the weighted bottleneck distance for all lines in a given equivalence class, once the equivalence classes are generated by the set of critical values \emph{and} switch points. In particular, the refinement of the arrangement $\mathcal{A}_0$ to $\mathcal{A}$ mirrors the refinement of the equivalence relation achieved when adding the set of switch points to the set of critical values and the diagonal direction (the line $m=1$ in the dual space is also added to $\mathcal{A}$). More about the difference between our switch points and the lines of the arrangement $\mathcal{A}$ will be discussed in \Cref{sec:rel-work4}.

Following \cite{Kerber-Lesnick-Oudot2018}, they are able to show that, on any face in the line arrangement, the weighted bottleneck distance obtains a supremum at a boundary vertex of this face, or as the limit of an outer segment (Lemma 10). The latter involves considering horizontal lines (in the primal plane), which are excluded by the approach followed here.

In summary, the main differences between \cite[Lemma 10]{Kerber-Lesnick-Oudot2018} and our \Cref{thm:main}, apart from the use of point-line duality, are: 
\begin{itemize}
    \item \Cref{thm:main} states that the matching distance is achieved as a maximum, while in \cite[Lemma 10]{Kerber-Lesnick-Oudot2018} it is achieved as a supremum (however easy to compute).
    \item \Cref{thm:main} avoids considering horizontal lines, i.e.~points $(m,q)$ with $m=0$ in the dual space (vertical lines are implicitly not needed also in \cite{Kerber-Lesnick-Oudot2018}).
\end{itemize}

Likewise, in \cite{Bjerkevik2021} the initial line arrangement of \cite{Kerber-Lesnick-Oudot2018} is refined but based on a real parameter $\lambda>0$: this time the
crucial property is that within each region in the dual plane  determined by the line arrangement, the weighted bottleneck distance  is either $\le \lambda$  or $> \lambda$.

\section{Existence of the switch points}\label{sec:switch}


The main result of this section is \Cref{thm:PimpliesQ}, stating  that, if $M$ and $N$ satisfy Property (P), then also Property (Q), and hence (Q'), is satisfied.  It is important to note that this result is not only an existence result  but leads to a  constructive procedure to determine  the set of switch points $\Omega(C_M,C_N)$.

Switch points are needed when there exist two matched pairs $u,v$ and $w,x$ with $u,v,w,x\in\closure C_M\cup \closure C_N$ for which, 
\begin{itemize}
\item there is a line $L$ in a particular equivalence class via $\sim_{\closure C_M\cup \closure C_N}$ on which the cost of matching $u$ with $v$ equals the cost of matching $w$ with $x$, but
\item the cost of matching $u$ with $v$ {\em does not} equal the cost of matching $w$ with $x$ for all lines in that equivalence class.
\end{itemize}

Therefore, we use \Cref{lem:omegapts} to determine when, given a quadruple $u,v,w,x$ and equivalence class $E$ via $\sim_{\closure C_M\cup \closure C_N}$, it is possible to have a line $L\in E$ with the aforementioned equality.  As the following proposition shows, all such lines with this equality in $E$ pass through some point $\omega\in\mathbb{P}^2$ (see \Cref{fig:omega4set}).  Spanning over all quadruples and all equivalence classes via $\sim_{\closure C_M\cup \closure C_N},$ we obtain the set of switch points $\Omega(C_M,C_N).$

To avoid degenerate or trivial cases, we assume that $u\neq v$ and $w\neq x$.  This is because the matched pair $u,v$ with $u=v$ would indicate a simultaneous birth and death, i.e.~a point on the diagonal within a persistence diagram that yields a null contribution to the matching distance (similarly for $w=x$). However, it can be the case that one of the 4 points is repeated, such as $v=x.$  In this case, a switch point would indicate a line $L$ for which the cost of matching $u$ to $v$ is the same as matching $w$ to $v$ (see \Cref{fig:omega3set}).

\begin{figure}[H]
	\begin{center}
 		 \begin{minipage}[c]{0.45\textwidth}
		 \begin{center}
			\begin{tikzpicture}
				\input{./figures/OmegaFrom3Points.tex}
			\end{tikzpicture}
			\caption{An example of 3 points $u,w,v=x$ that generate an $\omega$ point. Here, $m$ denotes the midpoint of the line segment between $v$ and $x$.  The pushes of the values $u$ and $w$ to $L$ are denoted as red points, and the push of $v$ is $\omega$ itself.  For any line which passes through $\omega$, the cost of matching $u$ with $v$ is the same as the cost of matching $w$ with $x$. }
			\label{fig:omega3set}
		\end{center}
		\end{minipage}\hfill
 		\begin{minipage}[c]{0.45\textwidth}
		\begin{center}
			\begin{tikzpicture}
				\input{./figures/OmegaFrom4Points.tex}
			\end{tikzpicture}
			\caption{An example of 4 distinct points $u,w,v,x$ which generate an $\omega$ point.  Here, $m_1$ and $m_2$ denote the midpoints of the line segments between $u$ and $w$ and $v$ and $x$, respectively.  The pushes of the values $u,w,v,x$ to $L$ are denoted as red points.  Since the two red segments in the figure have equal length for any line that passes through $\omega$, on those lines the cost of matching $u$ with $v$ is the same as the cost of matching $w$ with $x$. }
			\label{fig:omega4set}
		\end{center}
		\end{minipage}
	\end{center}
\end{figure}

Recall the notation $p_L(u)$  for
the parameter value of the push of a point $u$ onto the line $L$,  and recall that, for every $u\in  \RR^2$, $u$ pushes rightwards if $A^L_u=\{2\}$,  upwards if $A^L_u=\{1\}$, and to itself if  $A^L_u=\{1,2\}$. 
%
In \Cref{lem:omegapts} we will also use the following notation used to indicate sets of lines which partition a particular configuration of points in a particular way. 

\begin{notation}\label{notation:A(U)}  Let $(u,v,w,x)$ be a quadruple of points, and let $A_u,A_v,A_w,A_x\subseteq \{1,2\}$.  Define 
\[
\begin{split}
 E(A_u,A_v,A_w,A_x):=&\big\{L\textrm{ is a line with positive slope }\big|\\& A_u^L=A_u, A_v^L=A_v, A_w^L=A_w, A_x^L=A_x\big\},   
\end{split}
\]
i.e.~$E(A_u,A_v,A_w,A_x)$ is the set of lines with positive slope for which $u,v,w,x$ are separated by these lines in the way described by the subsets $A_u,A_v,A_w,A_x$.  Note that $E(A_u,A_v,A_w,A_x)$ is necessarily a union of equivalence classes of lines mod $\sim_{\closure C_M\cup \closure C_N}$.
\end{notation}

We can now prove the following result.

\begin{lemma}\label{lem:omegapts}
Let $u, v, w, x \in \closure C_M\cup \closure C_N$.
Assume that $u \neq v$ and $w \neq x$, and that there are at least three distinct points among $u,v,w,x$.
Let $\delta, \eta \in \{1,2\}$\footnote{In fact, $\delta$ and $\eta$ are determined by the pairs $u,v$ and $w,x$ respectively. For example, if both points $u,v$ are in $\closure C_M$, then this pair represents matched points in $\dgm{M^L}$ to the diagonal for some $L\in E$, and $\delta=2$. If $u\in\closure C_M$ and $v\in\closure C_N,$ then this pair represents the cost of matching a point in $\dgm M^L$ to $\dgm N^L$, and $\delta=1$.  The value of $\eta$ may be computed analogously.  However, we prove the proposition for the slightly more general case of arbitrary $\delta, \eta \in \{1,2\}$.}.
Fix a choice for each $A_u, A_v, A_w, A_x\subseteq\{1,2\}$, and set $E:=E(A_u, A_v, A_w, A_x)$
For any line $L \in E$,
consider the real quantity

\[\Delta_L(u,v,w,x;\delta,\eta) := \frac{|p_L(u) - p_L(v)|}{\delta} - \frac{|p_L(w)- p_L(x)|}{\eta}.\]

If $\Delta_L(u,v,w,x;\delta,\eta) = 0$, then exactly one of the following two possibilities is true, depending on $E$:
\begin{enumerate}
\item[(i)] We have $\Delta_{L'}(u,v,w,x;\delta,\eta) = 0$ for every $L' \in E$.
\item[(ii)] There exists a point $\omega \in \mathbb{P}^2$ such that, for every $L' \in E$, we have $\Delta_{L'}(u,v,w,x;\delta,\eta) = 0$ if and only if $\omega \in L'$.
\end{enumerate}
\end{lemma}

In other words, \Cref{lem:omegapts} says that, for fixed $A_u, A_v, A_w, A_x$, either the sign of $\Delta_L$ is constant on the equivalence classes of lines $L$ onto which the points  $u,v,w,x$ push as prescribed by the $A_i$'s, or otherwise the lines where there is a switch of the sign of $\Delta_L$ belong to a pencil  though a point $\omega \in \mathbb{P}^2$.

\begin{proof}

Consider some $L \in E$, parameterized as $\vec{m}s + b$.
The four points can be in various different positions with respect to $L$ and we want to find the lines for which the quantity
\[\Delta_L(u,v,w,x;\delta,\eta) :=  \frac{|p_L(u) - p_L(v)|}{\delta}- \frac{|p_L(w) - p_L(x)|}{\eta}\]
vanishes.  As $ \frac{|p_L(u) - p_L(v)|}{\delta}$  represents the cost of pairing $u$ with $v$ and $\frac{|p_L(w) - p_L(x)|}{\eta}$ represents the cost of pairing $w$ with $x$, respectively, this is exactly the set of lines on which these costs are equal.

Recall that saying that a point is on one side of a line includes the possibility that the point is on the line; saying that a point is \emph{strictly} on one side of the line excludes the possibility that the point is on the line.

Without loss of generality, we have the following possibilities.

\begin{enumerate}
    \item\label{case1} There exists $i\in\{1,2\}$ such that $i\in A_p$ for all $p\in\{u,v,w,x\}$. That is, all four points lie on one side of $L$ or, 
    \item\label{case2} There exists $i\in\{1,2\}$ such that $i\in A_p$ for exactly 3 of $p\in\{u,v,w,x\}$. That is, three points lie on one side of $L$, and the fourth lies strictly on the other side.
    \item\label{case3} Two paired points have $A_p=\{1\}$ for both points in the pair, and the other two paired points have $A_p=\{2\}$ for both points in the pair.  This is the case when two paired points lie strictly on one side of $L$, and the other two paired points lie strictly on the other side of $L$. 
    \item\label{case4} Two unpaired points have $A_p=\{1\}$, and the other two unpaired points have $A_p=\{2\}$. This is the case when two paired points lie strictly on opposite sides of $L$, and the other two paired points also lie strictly on opposite sides of $L$.
\end{enumerate}

We now address each possibility listed in  \cref{case1,case2,case3,case4} recalling \Cref{eq:param-1,eq:param-2,eq:param-3}:
\begin{enumerate}
    \item 
    Suppose, without loss of generality, that 
    \[p_L(u) \geq p_L(v) \text{ and } p_L(w) \geq p_L(x).\]
    Then the value of $\Delta_L$ is given by:
    \begin{align}\label{eq:omega1}
      &\Delta_L(u,v,w,x;\delta,\eta) = \frac{1}{\delta}\left(\frac{u_i}{m_i} - \frac{v_i}{m_i}\right) - \frac{1}{\eta}\left(\frac{w_i}{m_i} - \frac{x_i}{m_i}\right) \nonumber\\
      &=\left(\frac{u_i}{\delta} + \frac{-v_i}{\delta} + \frac{-w_i}{\eta} + \frac{x_i}{\eta}\right)\frac{1}{m_i},
    \end{align}
    with $i=1$ (resp. $i=2$) when all the points are to the right (resp. left) of $L$.
    
    \item Suppose without loss of generality that $u$, $v$, $w$ lie to the left of $L$ and  $x$ lies strictly to the right. Suppose also, without loss of generality, that $p_L(u) \geq p_L(v)$.
      
      If $p_L(w) \geq p_L(x)$, then
    \begin{align}\label{eq:omega2a}
    &\Delta_L(u,v,w,x;\delta,\eta) = \frac{u_2}{\delta m_2} - \frac{v_2}{\delta m_2} - \frac{(w_2 - b_2)}{\eta m_2} + \frac{(x_1 - b_1)}{\eta m_1} \nonumber\\
    &=\left(\frac{x_1}{\eta}\right)\frac{1}{m_1} +
    \left(\frac{u_2}{\delta} - \frac{v_{2}}{\delta} - \frac{w_2}{\eta}\right)\frac{1}{m_2} +
    \left(\frac{-1}{\eta}\right)\frac{b_1}{m_1} +
    \left(\frac{1}{\eta}\right)\frac{b_2}{m_2}.
    \end{align}
    
    If $p_L(x) \geq p_L(w)$, then
    \begin{align}\label{eq:omega2b}
      &\Delta_L(u,v,w,x;\delta,\eta) = \frac{u_2}{\delta m_2} - \frac{v_2}{\delta m_2} + \frac{(w_2 - b_2)}{\eta m_2} - \frac{(x_1 - b_1)}{\eta m_1} \nonumber\\
      &=\left(\frac{-x_{1}}{\eta}\right)\frac{1}{m_1} +
      \left(\frac{u_{2}}{\delta} - \frac{v_{2}}{\delta} + \frac{w_{2}}{\eta}\right)\frac{1}{m_2} +
      \left(\frac{1}{\eta}\right)\frac{b_1}{m_1} +
      \left(\frac{-1}{\eta}\right)\frac{b_2}{m_2}
    \end{align}

The case when   $u$, $v$, $w$ lie to the right of $L$ and  $x$ lies strictly to the left can be handled similarly.
    
    \item Suppose, without loss of generality, that $u$ and $v$ lie strictly to the left of $L$ and that $w$ and $x$ lie strictly to the right of $L$.
    Suppose also, without loss of generality, that $p_L(u) \geq p_L(v)$ and $p_L(w) \geq p_L(x)$.
    Then
    \begin{align}\label{eq:omega3}
      &\Delta_L(u,v,w,x;\delta,\eta) = \frac{u_2}{\delta m_2} - \frac{v_2}{\delta m_2} - \frac{w_1}{\eta m_1} + \frac{x_1}{\eta m_1} \nonumber\\
      &=\left(\frac{-w_{1}}{\eta} + \frac{x_{1}}{\eta}\right)\frac{1}{m_1} +
        \left(\frac{u_{2}}{\delta} - \frac{v_{2}}{\delta}\right)\frac{1}{m_2}.
    \end{align}
    
    \item Suppose,  without loss of generality, that $u$ and $w$ lie strictly to the left of $L$ and  that $v$ and $x$ lie strictly to the right of $L$. 
    Also suppose, for brevity, that $p_L(u) \geq p_L(v)$. The case when $p_L(u) \leq p_L(v)$ can be treated similarly and produces formulas for $\Delta_L(u,v,w,x;\delta,\eta)$ that differ only by the sign.
    If $p_L(w) \geq p_L(x)$, then
    \begin{align}\label{eq:omega4a}
      &\Delta_L(u,v,w,x;\delta,\eta) = \frac{u_2 - b_2}{\delta m_2} - \frac{v_1 - b_1}{\delta m_1} - \frac{w_2 - b_2}{\eta m_2} + \frac{x_1 - b_1}{\eta m_1}. \nonumber\\
      &=\left(\frac{-v_{1}}{\delta} + \frac{x_{1}}{\eta}\right)\frac{1}{m_1} +
        \left(\frac{u_{2}}{\delta} + \frac{-w_{2}}{\eta}\right)\frac{1}{m_2} +
        \left(\frac{1}{\delta} + \frac{-1}{\eta}\right)\frac{b_1}{m_1} +
        \left(\frac{-1}{\delta} + \frac{1}{\eta}\right)\frac{b_2}{m_2}.
    \end{align}
    If $p_L(x) \geq p_L(w)$, then
    \begin{align}\label{eq:omega4b}
      &\Delta_L(u,v,w,x;\delta,\eta) = \frac{u_2 - b_2}{\delta m_2} - \frac{v_1 - b_1}{\delta m_1} + \frac{w_2 - b_2}{\eta m_2} - \frac{x_1 - b_1}{\eta m_1} \nonumber\\
      &=\left(\frac{-v_{1}}{\delta} + \frac{-x_{1}}{\eta}\right)\frac{1}{m_1} +
        \left(\frac{u_{2}}{\delta} + \frac{w_{2}}{\eta}\right)\frac{1}{m_2} +
        \left(\frac{1}{\delta} + \frac{1}{\eta}\right)\frac{b_1}{m_1} +
        \left(\frac{-1}{\delta} + \frac{-1}{\eta}\right)\frac{b_2}{m_2}.
    \end{align}
\end{enumerate}

Observe that in all \Cref{eq:omega1,eq:omega2a,eq:omega2b,eq:omega3,eq:omega4a,eq:omega4b}, the expression for $\Delta_L(u,v,w,x;\delta,\eta)$ has the form 
\[\Delta_L(u,v,w,x;\delta,\eta) = \alpha_1 \cdot \frac{1}{m_1} + \alpha_2 \cdot \frac{1}{m_2} + \beta \cdot\frac{b_1}{m_1} - \beta \cdot \frac{b_2}{m_2},\]
where $\alpha_1$, $\alpha_2$, and $\beta$ are constants that depend only on $u,v,w,x$ and $\delta,\eta$ and not on the line $L$.

We now consider the following exhaustive cases:
\begin{itemize}
    \item $\beta =0$ and at least one of $\alpha_1,\alpha_2$ equal $0$,
    \item $\beta = 0$ and $\alpha_1,\alpha_2\neq 0$,
    \item $\beta\neq 0.$
\end{itemize}

Suppose first that $\beta$ is zero and one of $\alpha_1$ and $\alpha_2$ is zero. 
Assume, without loss of generality,  that $\alpha_1 = 0$.
Since $m_2 > 0$, then the sign of $\Delta_L(u,v,w,x;\delta,\eta)$ is the same as the sign of $\alpha_2$.
In other words, either $\Delta_L(u,v,w,x;\delta,\eta) = 0$ for every line $L \in E$, or it is positive for every line $L \in E$, or negative for every line $L \in E$. Thus, this case gives rise to no switch point $\omega$.

This implies that \Cref{eq:omega1} generates no switch points, as in that equation $\beta=0$ and one of $\alpha_1$ and $\alpha_2$ is zero.  We can conclude that when all points $u,v,w,x$ lie on the same side of any line $L$ in $E$, no switch point is created.

In \Cref{eq:omega3}, it is always the case that $\beta=0.$  If it is also the case that $x_1=w_1$ or $u_2=v_2$, then either $\alpha_1$ or $\alpha_2$ is zero.  If we assume that $x_1=w_1$, for $\Delta_L$ to be zero, it must also be the case that $u_2=v_2.$  This implies that $c_L(u,v)=c_L(w,x)=0$ for all lines $L\in E,$ and therefore $\Delta_L(u,v,w,x;\delta,\eta) = 0$ for every line $L \in E$.  (An identical argument works for $u_2=v_2$ as well.)  Thus, no switch point is created.

It is also possible to have $\beta=0$ in \Cref{eq:omega4a}, when $\delta=\eta.$  If, additionally, $x_1=v_1$ or $u_2=w_2$, then either $\alpha_1$ or $\alpha_2$ is zero. However, if $x_1=v_1$, then, for $\Delta_L$ to be zero, it must also be the case that $u_2=w_2.$  This implies that $c_L(u,v)=c_L(w,x)$ for all lines $L\in E,$ and therefore $\Delta_L(u,v,w,x;\delta,\eta) = 0$ for every line $L \in E$.  (An identical argument works for $u_2=w_2$ as well.) So no switch point is created.  To summarize, no switch points are created for:
\begin{itemize}
    \item \Cref{eq:omega1}
    \item \Cref{eq:omega3}, when $x_1=w_1$ or $u_2=v_2,$ nor
    \item \Cref{eq:omega4a}, when $\delta=\eta$ and $x_1=v_1$ or $u_2=w_2$.
\end{itemize}

Next suppose that $\beta = 0$ and $\alpha_1 \neq 0$ and $\alpha_2 \neq 0$.
In this case, the quantity $\Delta_L(u,v,w,x;\delta,\eta)$ is zero if and only if 
\[m_2 \cdot \alpha_1 = - m_1 \cdot \alpha_2.\]
This is a condition on the slope of the line, which states that $L$ must have a fixed slope.  When considering the line $L$ in $\mathbb{P}^2$, this means that $L$ passes through the point at infinity $\omega=[0:-\alpha_1:\alpha_2]$. Conversely, it is easy to check that any line $L \in E$ with this slope has $\Delta_L(u,v,w,x;\delta,\eta) = 0$.

It is only possible for $\beta=0$ in \Cref{eq:omega3} and in \Cref{eq:omega4a}, when $\delta=\eta$.  In both of these, to ensure that $\alpha_1 \neq 0$ and $\alpha_2 \neq 0$, it must be that $x_1\neq w_1$ and $u_2\neq v_2.$ 

Using this formula, we obtain the following $\omega$ points on the line at infinity:

\begin{itemize}
    \item From \Cref{eq:omega3}, when $x_1\neq w_1$ and $u_2\neq v_2$:
    \begin{equation}\label{eq:omega-ter}
    \omega=\left[0:\delta(w_1-x_1):\eta(u_2-v_2)\right].
    \end{equation}

    \item From \Cref{eq:omega4a}, when $\delta=\eta$ and $x_1\neq v_1$ and $u_2\neq w_2$:
	  \begin{equation}\label{eq:omega-quater}     
     \omega=\left[0: v_1- x_1: u_2 - w_2\right]. 
     \end{equation}
    
    \end{itemize}
 
Finally, suppose that $\beta \neq 0$.
In this case, $\Delta_L(u,v,w,x;\delta,\eta) = 0$ if and only if 
\[\frac{\alpha_1}{\beta} \cdot \frac{1}{m_1} + \frac{\alpha_2}{\beta}\cdot \frac{1}{m_2} + \frac{b_1}{m_1} - \frac{b_2}{m_2} = 0.
\] 
Again, viewing the lines $L$ in $\mathbb{P}^2$, it is easy to check that this equation is satisfied  if and only if the line $L\in E$ passes through the point 
\[\omega = \left[\beta:-\alpha_1:\alpha_2\right].\]

When we require that $\beta\neq0$, we only need to consider \Cref{eq:omega2a,eq:omega2b,eq:omega4a} (when $\delta\neq\eta$), and \Cref{eq:omega4b}. 

Using this formula, we obtain the following proper $\omega$ points:
\begin{itemize}
    \item  From \Cref{eq:omega2a}:
    \begin{equation}\label{eq:omega-1}
    \omega=\left[\delta:\delta x_1:\eta(v_2-u_2)+\delta w_2\right].    
   \end{equation}
    \item From \Cref{eq:omega2a}, altered to consider the case when $u,v,$ and $w$ lie to the right of lines $L\in E$, and $x$ to the left:
    \begin{equation}\label{eq:omega-2}
    \omega=\left[\delta:\eta(v_1-u_1)+\delta w_1:\delta x_2\right].
    \end{equation}
    
    \item From \Cref{eq:omega2b}:
    \begin{equation}\label{eq:omega-3}
	\omega=\left[\delta:\delta x_1:\eta(u_2-v_2)+\delta w_2\right].
      \end{equation}
     
   \item From \Cref{eq:omega2b}, altered to consider the case when $u,v,$ and $w$ lie to the right of lines $L\in E$, and $x$ to the left:
 \begin{equation}\label{eq:omega-4}
 \omega=\left[\delta:\eta(u_1-v_1)+\delta w_1:\delta x_2\right].
  \end{equation}

    \item From \Cref{eq:omega4a} when $\delta\neq\eta$:
	\begin{equation}\label{eq:omega-5}    
    \omega=\left[\eta-\delta:\eta v_1 - \delta x_1:\eta u_2 - \delta w_2\right].
     \end{equation}
    
    \item From \Cref{eq:omega4b}:
	\begin{equation}\label{eq:omega-6}    
    \omega=\left[\eta+\delta:\delta x_1+\eta v_1:\delta w_2+\eta u_2\right].
     \end{equation}
     
    \end{itemize}

\end{proof}

\begin{remark}\label{rem:no-need-closure}
\Cref{lem:omegapts} describes how to construct certain $\omega$ points based on a choice of a quadruple $u,v,w,x \in \overline{C}_M \cup \overline{C}_N$ and sets $A_p$ for $p \in \{u,v,w,x\}$.
In fact, the same $\omega$ points can be obtained if we choose $\{u,v,w,x\}$ from the more restricted set $C_M \cup C_N$.

For example, let $u \in (\overline{C}_M \cup \overline{C}_N \setminus (C_M \cup C_N)$. 
Then $u$ is the least upper bound either of two points in $C_M$, or two points in $C_N$.
Suppose, without loss of generality, that $u', u'' \in C_M$ such that $u = ((u')_1,(u'')_2)$.
If $L$ is a line such that the push direction of $u$ onto $L$ is upward, then $p_L(u) = p_L(u')$.
Similarly, if $L$ is a line such that the push direction of $u$ onto $L$ is rightward, then $p_L(u) = p_L(u'')$.
Since the formula for $\omega$ only depends on the direction in which $u$ pushes, the value of $\omega$ does not change.
\end{remark}

The proof of next \Cref{thm:PimpliesQ} uses \Cref{lem:omegapts} to generate a set $\Omega(C_M,C_N)\subset\mathbb{P}^2$ of switch points for which Property (Q') holds.  We refer the reader to \cite{brooks2023switch} for how to turn the ideas of this proof into an algorithm that computes $\Omega(C_M,C_N)$.

\begin{theorem}\label{thm:PimpliesQ}
 Given two 2-parameter persistence modules $M,N$ satisfying Property (P), there exists a set of  switch points $\Omega(C_M,C_N)$ ensuring that $M$ and $N$ also satisfy Property (Q), and hence (Q').
\end{theorem}

\begin{proof}
As usual, let $C_M$ and $C_N$ denote the set of critical values of $M$ and $N$, respectively, as guaranteed by Property (P), and let $\closure C_M$ and $\closure C_N$ be their closure with respect to the least upper bound.

For each quadruple $u,v,w,x\in \closure C_M\cup \closure C_N$ (or simply each quadruple in $C_M\cup C_N$, thanks to \Cref{rem:no-need-closure}) and choice of $A_u, A_v, A_w, A_x\subseteq\{1,2\}$ satisfying the conditions of \Cref{lem:omegapts}, we may obtain a switch point $\omega$, either proper or at infinity.  Iterating over all such quadruples and possible push directions for each quadruple, we obtain a finite set of points in $\mathbb{P}^2$, which we denote $\Omega(C_M,C_N)$, or briefly $\Omega$.  

We now show that the set $\Omega$ satisfies Property (Q); then, by \Cref{prop:Q}, it satisfies Property (Q').  

Let $L$ and $L'$ be two lines such that $L\sim_{\closure C_M\cup \closure C_N\cup{\Omega}}L'$, and let $u,v,w,x$ be any quadruple, with at least three distinct points, in $\closure C_M\cup \closure C_N$. We want to prove that  the signs of $\Delta_L(u,v,w,x;\delta,\eta)$ and  $\Delta_{L'}(u,v,w,x;\delta,\eta)$ agree. 

First, by \Cref{lem:omegapts}, if $\Delta_L(u,v,w,x;\delta,\eta)=0$, then either $\Delta_{L'}(u,v,w,x;\delta,\eta)=0$ for every $L'$ in the same equivalence class, or all lines in the same equivalence class  contain a switch point  $\omega\in\Omega$.
So, by the definition of the equivalence relation, for any $L'\sim_{\closure C_M\cup \closure C_N\cup{\Omega}}L$, it must also be the case that $\omega\in L'$.  By \Cref{lem:omegapts}, this implies that $\Delta_{L'}(u,v,w,x;\delta,\eta)=0$ as well.

Next suppose, without loss of generality, that 
\[\Delta_L(u,v,w,x;\delta,\eta)>0 \textrm{ and } \Delta_{L'}(u,v,w,x;\delta,\eta)<0.\]
This implies that neither $L$ nor $L'$ contain the point $\omega$ corresponding to the quadruple $u,v,w,x$ and to the choice of $A_u, A_v, A_w, A_x\subseteq\{1,2\}$ which generate $\omega$. 

Note that, for fixed $\delta$ and $\eta$ and within an equivalence class with respect to $\closure C_M\cup \closure C_N\cup{\Omega}$, the quantity $\Delta_L(u,v,w,x;\delta,\eta)$ is a continuous function with respect to translation or rotation. Moreover, $L'$ may be obtained from $L$ by a translation and/or rotation which is always in the same equivalence class with respect to $\closure C_M\cup \closure C_N\cup\Omega$. 

Therefore, $\Delta_L(u,v,w,x;\delta,\eta)$ is a continuous function of $L$ which changes sign within an equivalence class with respect to $\closure C_M\cup \closure C_N\cup{\Omega}$, and by the Intermediate Value Theorem, there must be a line $L''$, equivalent to both $L$ and $L'$ via $\sim_{\closure C_M\cup \closure C_N\cup{\Omega}}$ such that $\Delta_{L''}(u,v,w,x)=0$.  This implies that $\omega\in L''$, and since $L''\sim_{\closure C_M\cup \closure C_N\cup{\Omega}} L'$ and $L''\sim_{\closure C_M\cup \closure C_N\cup\Omega} L$, then $\omega\in L'\cap L$. 

This is a contradiction, and therefore, the sign of $\Delta_L(u,v,w,x;\delta,\eta)$ is the same as the sign of $\Delta_{L'}(u,v,w,x;\delta,\eta)$ for all lines $L\sim_{\closure C_M\cup \closure C_N\cup\Omega}L'$.  This must be true for all quadruples in $\closure C_M\cup \closure C_N$, so Property (Q) holds for $\Omega.$
\end{proof}

\subsection{Connections to related works}\label{sec:rel-work4}

The need to consider switch points is ubiquitous in the literature about the computation of the matching distance. To the best of our knowledge, however, switch points' coordinates are always left implicit and never worked out. See, for example, \cite[Definition 5.1]{ethier2023geometry} and equations (1) in \cite{Kerber-Lesnick-Oudot2018}. 

Working out the explicit formulas, as done in the proof of \Cref{lem:omegapts}, we found it convenient to pass to the projective plane because some switch points belong to the line at infinity. In other words, a switching may happen when lines pass through a given direction rather than a given point. In the dual plane, this corresponds to a point passing through a vertical line. To the best of our understanding, this case is not singled out in the arrangement $\mathcal{A}$ of \cite{Kerber-Lesnick-Oudot2018}. The only vertical line explicitly discussed in the dual plane is the one corresponding to the diagonal direction in the primal plane.  

Interestingly, vertical lines in the dual plane play a role in \cite{Bjerkevik2021}: The change through the value $\lambda$ of the weighted bottleneck distance is shown to happen at a dual point on a line in dual space that can be
vertical and therefore does not have a corresponding point in primal space (if not at infinity).


\section{The matching distance algorithm}\label{sec:match_dist_alg}

The algorithm to compute the matching distance between the persistence modules $M$ and $N$  based on the \Cref{thm:main} can be now described by the following main steps:

\begin{enumerate}
    \item Determine the sets of critical values $C_M$ and $C_N$ for $M$ and $N$, respectively (for example, using the results of \cite{Bapat2022});
    \item Identify the closures $\closure C_M$ and $\closure C_N$;
    \item Compute the set of switch-points $\Omega=\Omega(C_M,C_N)$ guaranteed by \Cref{thm:main};
    \item For each line $L$ with positive slope that passes through two distinct points, one in $\closure C_M\cup \closure C_N \cup \Omega_P$ and the other one in $\closure C_M\cup \closure C_N \cup \closure\Omega$ do:
    \begin{enumerate}
    \item Compute the persistence diagrams $\dgm(M^L)$, $\dgm(N^L)$ of the restrictions of $M$ and $N$ to that line $L$;
    \item Compute their weighted bottleneck distance $\hat m^L d_B(\dgm(M^L),\dgm(N^L))$;
    \end{enumerate}
    \item Output the maximum of the weighted bottleneck distances computed in Step 4(b).   
\end{enumerate}

To analyze the complexity of such an algorithm, let $m=|C_M|$, $n=|C_N|$, $h=|\Omega|$. For simplicity, we also set $r=m+n$ and $p=m+n+h$. In the worst case,  the order of $|\Omega|$ is   $O(r^4)$  even if in practical cases this number is much smaller \cite{brooks2023switch}.  Thus, in the worst case, the order of $p$ is $O(r^4)$.

The  time complexity  of computing the least upper bound of two points of $\RR^2$ is $O(1)$, so doing it for all pairs of points of a set with $n$ points is $O(n^2)$. From this we can deduce the complexity of step 2 in terms of $m$ and $n$ is $O(m^2+n^2)$.

The time complexity of computing all the possible lines through two points in a set of $p$ points is $O(p^2)=O(r^8)$.

The time complexity of computing the persistence diagram of the restriction of $M$ and $N$ to a line is $O(m^\mu+n^\mu)$, where $\mu\approx 2.37$ is the matrix multiplication constant. That of computing the bottleneck distance in a set with $s$ points of the plane is $O(s^{\mu/2}\log s)$ \cite{Katz2023}. For us, $s$ can be assumed to be in the order of $m+n$, yielding $O((m+n)^{\mu/2}\log(m+n))$ for the cost of computing the bottleneck distance for a line. Thus, the sum of the costs of steps (a) and (b) can be evaluated to be no more than $O(r^\mu)$.

Because this has to be repeated for each line, whose number is $O(r^8)$, we get the total time complexity for step 4 bounded by $O(r^{8+\mu})$, which dominates that of the previous steps. In conclusion, the time complexity of our method is roughly $O(r^{11})$.

\subsection{Connections to related works}

The algorithm we streamlined above mirrors, in the primal plane, the algorithm of \cite{Kerber-Lesnick-Oudot2018}. The differences are that we manage not to consider horizontal lines, and that we have explicit formulas for switch points. Computationally, the running time agrees with that of \cite{lesnick2022computing} and is higher than that of \cite{Bjerkevik2021} which is $O(r^5\log^3r)$.

We note that the matching distance computation is parallelizable, with bottleneck distance computations carried out in parallel on different lines. Additionally, the computation of candidate switch points is also parallelizable based on the cases described in the proof of \Cref{lem:omegapts}, \cref{case2} through \cref{case4}.   
 Algorithms to prune the set of candidate switch points are detailed in \cite{brooks2023switch}.

\section{Conclusions}\label{sec:conclusions}

The main contributions of this paper are:  
\begin{itemize}

    \item Theoretical results for computing the matching distance between two $2$-parameter persistence modules, in a combinatorial setting and while staying in the primal plane;
    
     \item New definitions for the costs (w.r.t.~computing the matching distance) of horizontal and vertical lines of a $2$-parameter persistence module; and
     
    \item The identification of the parameter space of a $2$-parameter persistence module as a subset of the projective plane.

\end{itemize}


We expand on these bullets in order.  First, one advantage of the theoretical framework developed in this paper is that it allows us to formulate explicit ways of computing the switch-points from the critical values of two $2$-parameter persistence modules (see \cite{brooks2023switch}). This allows us to prove that the matching distance in two dimensions as the maximum over a finite set of lines, as guaranteed by \Cref{thm:main}. 

Second, in addition to addressing the explicit computation of the matching distance, we provide a detailed study of the lines through the parameter space and a geometric interpretation corresponding to a line being vertical, horizontal, or diagonal. 

About diagonal lines, we show via \Cref{ex:need-diag} that in some cases, lines through pairs of critical points are not sufficient in order to compute the matching distance, but adding lines through critical points with the diagonal direction may turn out to be sufficient. This agrees with the observation in \cite{Cerri2019geometrical} that diagonal lines play a key role in the study of $2$-parameter persistence.  However, \Cref{ex:diag-not-suff,ex:need-omega} show that adding in diagonal lines is also not always sufficient as lines through critical and switch points need to be considered. While critical points are always proper points, our theoretical results show that some switch points are points at infinity of the projective plane. This way, it is possible to limit the number of slopes of lines on which the matching distance can be obtained. 

Regarding vertical and horizontal lines, we give an explicit formula for the computation of the weighted bottleneck distance along horizontal and vertical lines.   Moreover, we conclude that these lines can be excluded from the computation of the overall matching distance between two $2$-parameter persistence modules --- thanks to the inclusion of diagonal lines in the computation, there is always a line with positive but finite slope for which the weighted bottleneck distance is greater than that along a horizontal or vertical line. Viewing the parameter space as a subset of the projective plane helps in understanding the importance of vertical, diagonal, and horizontal lines and makes the treatment of these lines more consistent.

Third, in \cite{Kerber-Lesnick-Oudot2018,Bjerkevik2021}, the authors leverage the duality between points and lines in the plane to compute the matching distance, performing computations in the dual plane. In comparison, our method is directly defined in the parameter space of the persistence modules, helping the understanding of the finite set of lines involved in the maximization for the computation of the matching distance.  In particular, identifying the parameter space with a subset of the projective plane helps to highlight the importance of both proper and infinite switch points (switch points that represent positive slopes).

As a final note: our approach to the computation of the matching distance remains possible only in two dimensions.  This is due to the fact that, in two dimensions, a line is caged by a set of points, in our case the critical and switch points, whereas this is no longer the case in higher dimensions. For this reason, the method we propose does not, for now, generalize to more parameters.

So, In any case, it remains as an open question whether the matching distance for $n$-parameter persistence modules with $n\ge 3$ is achieved as a maximum.


\section*{Acknowledgements}
\noindent A.B. acknowledges the support of ICERM at Brown University, where part of this work was done.

\noindent R.B. was supported for travel to BIRS by an AWM-NSF Travel Grant.

\noindent C.H.\ is supported by NCCR-Synapsy Phase-3 SNSF grant number 51NF40-185897.

\noindent C.L. carried out this work under the auspices of INdAM-GNSAGA and partially within the activities of ARCES (University of Bologna).

\noindent B.I.M.~is supported by Digital Futures (dBrain).

\noindent E.S. was supported  for travel to BIRS by European Research Council (ERC) grant no.\ 788183, `Alpha Shape Theory Extended'.

\noindent We are grateful to the Banff International Research Station for hosting us as a Focused Research Group. 

\noindent We also thank Martin Hafskjold Thoresen for help with computations, and the anonymous referee who pointed out useful connections to the literature.

\printbibliography
%

\end{document}